\documentclass[11pt]{article}
\usepackage{stmaryrd}
\usepackage{epic,eepic}
\usepackage{geometry}

\usepackage{float}
\usepackage{amsmath}
\usepackage{amssymb}
\usepackage{amsthm}
\usepackage{graphicx} 
\usepackage{psfrag} 
\usepackage{enumerate}

\newcommand{\toi}{\mathcal{T}_i}

\def\ds{\displaystyle}

\def\O{\Omega}
\def\Oun{\Omega_1}
\def\Odeux{\Omega_2}
\def\G{\Gamma}

\def\dt{\delta t}

\def\R{\mathbb{R}}
\def\T{\mathcal{T}}
\def\P{\mathcal{P}}
\def\U{\mathcal{U}}
\def\Q{\mathcal{Q}}
\def\E{\mathcal{E}}
\def\K{\mathcal{K}}

\def\N{\mathcal{N}}
\def\eps{\varepsilon}
\def\meas{\text{meas}}
\def\diam{\text{diam}}

\def\demi{\frac{1}{2}}
\def\3demi{\frac{3}{2}}
\def\sumi{\sum_{i=1}^2}
\def\tp{\tilde{p}}
\def\tu{\tilde{u}}

\def\dxy{d(x_K,y_\eps)}
\def\sumk{\sum_{k=1}^{\K}}
\def\sumeps{\sum_{\eps\in\E}}
\def\jp12{{j+\demi}}
\def\jm12{{j-\demi}}

\newcommand\dpdt[1]{\frac{\partial p^{#1}}{\partial t}}
\newcommand\dpdn[1]{\frac{\partial p^{#1}}{\partial n_{#1}}}
\newcommand\sumn[1]{\sum_{n=0}^{N_{#1}-1}}
\newcommand\sumK[1]{\sum_{K \in \T_{#1}}}
\newcommand\sumKp[1]{\sum_{K' \in \mathcal{N}_{#1}(K)}}
\newcommand\sumE[1]{\sum_{\eps \in \E_{#1}(K)}}

\newcommand\sumEDE[1]{\sum_{\eps \in \E_{#1}(K)\cup\E_{#1D}(K)}}

\def\lprod{\left\langle}
\def\rprod{\right\rangle_{L^2(0,T;L^2(\G))}}

\newtheorem{assumption}{Assumption}[section]
\newtheorem{notation}[assumption]{Notation}
\newtheorem{theorem}[assumption]{Theorem}
\newtheorem{lemma}[assumption]{Lemma}
\newtheorem{definition}{Definition}[section]

\title{Local time steps for a finite volume scheme}

\author{
 I.~Faille, \thanks{ Institut Français du P\'etrole, Rueil-Malmaison, France. 
  isabelle.faille@ifp.fr}
  \and 
  F.~Nataf\thanks{Laboratoire J.L.~Lions, CNRS UMR7598, Universit\'e Pierre et Marie Curie, France. 
  nataf@ann.jussieu.fr}
  \and
   F.~Willien \thanks{Institut Français du P\'etrole, Rueil-Malmaison, 
   France. francoise.willien@ifp.fr}
  \setcounter{footnote}{6}
  \setcounter{footnote}{2}
  \and
  \and S.~Wolf\thanks{Laboratoire de Tectonique - CNRS UMR 7072, Universit\'e Pierre \& Marie Curie - Paris VI, 4 place Jussieu, 75252 Paris Cedex 05, France. sylvie.wolf@upmc.fr }
}

\allowdisplaybreaks[1]

\begin{document}

\maketitle

\tableofcontents

\begin{abstract}
We present a strategy for solving time-dependent problems on grids with local refinements in time using different time steps in different regions of space. We discuss and analyze two conservative approximations based on finite volume with piecewise constant projections and domain decomposition techniques. Next we present an iterative method for solving the composite-grid system that reduces to solution of standard problems with standard time stepping on the coarse and fine grids. At every step of the algorithm, conservativity is ensured. Finally, numerical results illustrate the accuracy of the proposed methods. 
\end{abstract}


\section{Introduction}
\label{sec:introduction}
In many physical applications, there are special features which greatly affect the solution globally
 as well as locally. One important example is the local spatial and temporal behavior of multiphase fluid flow
around a production well in the petroleum recovery applications. To capture this local behavior, spatial local
refinement is necessary. However, it requires a reduction of the time step, compared to the 
one used with a coarse mesh, in order to get a solution accurate enough in the refined zone and to avoid 
convergence problems when solving the non linear discretized equations. When applied uniformly on all the 
simulation domain, this reduced time step leads to unacceptable cpu-time making the use of local time steps  highly desirable. 
To be efficient, a  local time-stepping strategy (numerical scheme and solution method) must
\begin{itemize}
\item  ensure accuracy of the solution i.e. the solution has to be more accurate than the one obtained with 
a global coarse mesh,
\item ensure stability without any too restrictive condition on the time step,
\item  lead 
to reduced cpu-time compared to the one obtained when using a  small time step on the whole domain.
\end{itemize}
 In the framework of reservoir simulation where local grid refinement is  necessary to represent
 correctly important local phenomena in the wells vicinity, the corresponding numerical 
scheme must also be locally conservative in order to be applicable to multi-phase flow 
simulations where  a coupled system of parabolic and hyperbolic equations has to be solved.

For parabolic equations, different approaches have been proposed in the past which extend the classical
implicit finite difference scheme to local refinement in time. In \cite{ELV90}, the scheme is written as a
cell centered Finite Volume scheme. At the interface between the coarse time step zone and the refined time
step one, the flux over the coarse step is taken equal to the integral over the corresponding refined steps of the flux
computed from the refined zone. This refined flux approximation requires values of the unknowns in the coarse
zone at small time steps which are computed using  piecewise constant or linear interpolation from the coarse unknowns. The scheme
is conservative. Stability and error estimation are obtained for the piecewise constant interpolation. 
On the contrary, for the linear interpolation, stability is obtained under a sufficient condition which is as restrictive
as the time step limitation obtained for an explicit scheme. In \cite {DDD91},  Dawson et al  proposed to couple classical
 implicit finite difference
schemes in the refined and coarse zones using an explicit approximation at the interface on a larger
 mesh size  in order to attenuate the time step limitation due to the explicit approximation.
Although interesting for its simplicity, this approach can not be retained due to its time
step limitation.  In \cite{EL94}, Ewing and Lazarov  proposed an implicit non conservative approach.
 The scheme for the coarse nodes is straightforward  while the fine grid nodes located at the
 interface between the coarse and fine regions  require ``slave'' points at small time steps on
 the coarse grid side, which are not grid points.
As in \cite{ELV90}, the values of the unknown at these slave points are obtained by linear interpolation in time 
between the corresponding nodes of the coarse grid, and the set of discretized equations involves all the unknowns
between two coarse time levels.  Stability and error analysis are performed. The solution method uses
 an iterative method associated to 
 a coarse grid preconditionner of the Schur complement of the system where the refined region unknowns have 
been eliminated. 
In the more applied framework of compositional multiphase flow, \cite{DKH95} introduced
 an implicit timestepping method. For each global time step, the problem is solved implicitly
 in the whole domain but using a linear approximation of the model in the refined regions which avoids any
convergence problem of the non linear solver due to refined mesh. Then, the 
refined zones are solved using a local time step and taking  as boundary conditions the fluxes computed
 during the first stage at
the interface between the refined and coarse zones. This approach ensures that
the method is conservative. It is moreover rather efficient as, compared to the cpu-time necessary to
 solve the problem with a large timestep on the whole domain,  it only requires additional cpu-time to solve 
the equations once in the refined zone. However, the accuracy of the solution is not controlled. 
Looking for an  efficient solution method, \cite{SV2000} used the same finite difference scheme as \cite{EL94}
for linear parabolic equations but proposed a predictor-corrector method. In the predictor stage, the
 solution is computed at the coarse time step on all the domain and in the correction step, the solution
 is computed in the refined grid at small time steps using values at slave nodes interpolated from
the coarse nodes solution obtained in the first stage. They show that the predictor corrector approach preserves the
maximum principle satisfied by the solution of the scheme.

Our paper proposes a  local time step strategy based on the domain decomposition framework. It
 extends the approach introduced in  \cite{ELV90} by generalizing the interface conditions used to couple the coarse and refined time-step domains.  The method is conservative. Stability and
error estimates, which are different from that obtained in \cite{ELV90},  are presented.  A solution method which improves the 
predictor-corrector methods of  \cite{DKH95} and  \cite{SV2000} is proposed. 
In order to simplify the presentation and to concentrate on the difficulties arising from the local 
refinement in the time direction, we will first explain the approach in the case of a one-dimensional 
spatial problem in section \ref{section2}. Stability and error estimates are then proven in the more general case of nD spatial grids in sections 3 to 6. 
Finally, some numerical results are presented in section 7. 

\section{Description of the local time stepping strategy \label{section2}}
We consider the following problem:
Let $T>0$ and $\Omega$ be an open bounded domain of $\R^d$, $d \geqslant 1$, $p_0:\Omega\mapsto \R$ and $f:  \Omega \times (0,T)\mapsto \R$ be given functions.
Find $p : \Omega \times [0,T] \mapsto \R$ such that:
\begin{equation} \label{pb_cont}
\left.
\begin{aligned}
\dpdt{}(x,t) - \Delta p(x,t) = f(x,t) \quad \forall x \in \O \quad \forall t \in [0,T] \\
p(0,x) = p_0(x) \quad \forall x \in \O \\
p(x,t) = 0 \quad \forall x \in \partial \O \quad \forall t \in [0,T] 
\end{aligned}
\qquad\right\}
\end{equation}

In order to explain the scheme, we consider the $d=1$ case  with the cell centered grid shown in Figure \ref{figure:maillage1D}
and a time step which is variable in space. Namely, the domain $\O$ is decomposed into two non overlapping subdomains $\Oun$ and $\Odeux$ where  two different time-step sizes are used : the coarser time
 step is denoted $\dt_2$ (in $\Odeux$) and the finer time step is denoted $\dt_1$ (in $\Oun$) such that $\K \dt_1 = \dt_2$ with $\K \in \mathbb{N}^*$ (Figure  \ref{discretisation}).

\begin{figure}[hbt]
\begin{center}
\setlength{\unitlength}{0.00083333in}
\begingroup\makeatletter\ifx\SetFigFont\undefined%
\gdef\SetFigFont#1#2#3#4#5{%
  \reset@font\fontsize{#1}{#2pt}%
  \fontfamily{#3}\fontseries{#4}\fontshape{#5}%
  \selectfont}%
\fi\endgroup%
{\renewcommand{\dashlinestretch}{30}
\begin{picture}(6778,1060)(0,-10)
\put(148,549){\ellipse{70}{70}}
\put(414,549){\ellipse{70}{70}}
\put(668,556){\ellipse{70}{70}}
\put(921,549){\ellipse{70}{70}}
\put(1180,549){\ellipse{70}{70}}
\put(1460,549){\ellipse{70}{70}}
\put(1700,549){\ellipse{70}{70}}
\put(1960,549){\ellipse{70}{70}}
\put(2219,556){\ellipse{70}{70}}
\put(4146,540){\ellipse{70}{70}}
\put(5192,547){\ellipse{70}{70}}
\put(6247,548){\ellipse{70}{70}}
\put(2594,556){\ellipse{70}{70}}
\put(3247,548){\ellipse{70}{70}}
\path(798,710)(798,386)
\path(278,703)(278,374)
\path(538,703)(538,374)
\path(1317,703)(1317,374)
\path(1577,703)(1577,374)
\path(1832,703)(1832,374)
\path(2092,703)(2092,374)
\path(2356,703)(2356,374)
\thicklines
\path(853.930,844.950)(802.000,806.000)(853.930,767.050)
\path(802,806)(1062,806)
\path(1010.070,767.050)(1062.000,806.000)(1010.070,844.950)
\thinlines
\path(17,550)(6766,546)
\thicklines
\path(6752,805)(6752,286)
\path(22,807)(22,288)
\thinlines
\path(4673,694)(4673,365)
\path(3634,694)(3634,365)
\path(5712,694)(5712,365)
\path(1058,703)(1058,374)
\thicklines
\path(995.930,759.950)(944.000,721.000)(995.930,682.050)
\path(944,721)(1204,721)
\path(1152.070,682.050)(1204.000,721.000)(1152.070,759.950)
\path(2842,705)(2842,376)
\put(827,307){\makebox(0,0)[lb]{\smash{{{\SetFigFont{10}{12.0}{\rmdefault}{\mddefault}{\updefault}$x_j$}}}}}
\put(821,913){\makebox(0,0)[lb]{\smash{{{\SetFigFont{10}{12.0}{\rmdefault}{\mddefault}{\updefault}$h_j$}}}}}
\put(1143,793){\makebox(0,0)[lb]{\smash{{{\SetFigFont{10}{12.0}{\rmdefault}{\mddefault}{\updefault}$h_\jp12$}}}}}
\put(2747,772){\makebox(0,0)[lb]{\smash{{{\SetFigFont{10}{12.0}{\rmdefault}{\mddefault}{\updefault}$x_{I+\demi}$}}}}}
\put(2477,337){\makebox(0,0)[lb]{\smash{{{\SetFigFont{10}{12.0}{\rmdefault}{\mddefault}{\updefault}$x_I$}}}}}
\put(3117,346){\makebox(0,0)[lb]{\smash{{{\SetFigFont{10}{12.0}{\rmdefault}{\mddefault}{\updefault}$x_{I+1}$}}}}}
\put(967,68){\makebox(0,0)[lb]{\smash{{{\SetFigFont{12}{14.4}{\rmdefault}{\mddefault}{\updefault}$\Omega_1$}}}}}
\put(4367,58){\makebox(0,0)[lb]{\smash{{{\SetFigFont{12}{14.4}{\rmdefault}{\mddefault}{\updefault}$\Omega_2$}}}}}
\end{picture}
}
\caption{1D cell-centered grid }
    \label{figure:maillage1D}
\end{center}
\end{figure}
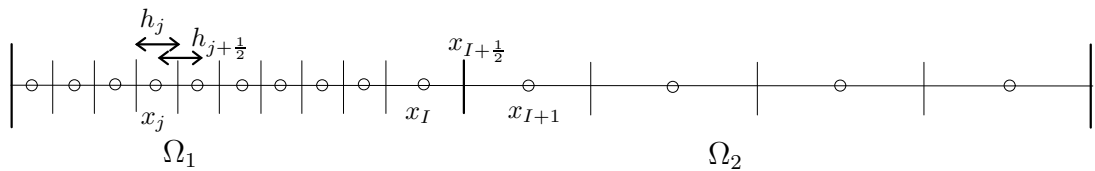

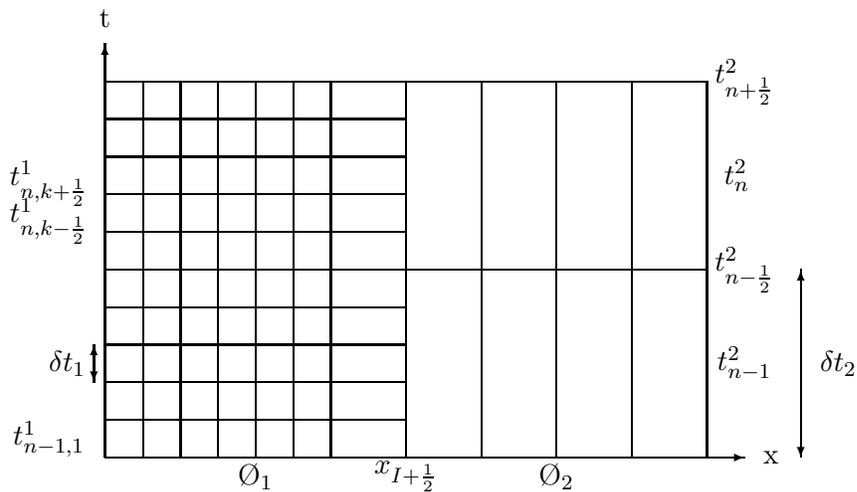
\begin{figure}[hbt]
\centering
\unitlength .5mm
\begin{picture}(160,150)

\matrixput(0,20)(1,0){1}(0,10){11}{\line(1,0){80}}
\matrixput(80,20)(1,0){1}(0,50){3}{\line(1,0){80}}

\matrixput(0,20)(10,0){7}(0,10){1}{\line(0,1){100}}
\matrixput(80,20)(20,0){5}(0,50){1}{\line(0,1){100}}

\put(160,20){\vector(1,0){10}}
\put(172,15){\makebox(10,10){x}}
\put(0,120){\vector(0,1){10}}
\put(-5,132){\makebox(10,10){t}}

\put(-20,20){\makebox(10,10){$t_{n-1,1}^1$}}
\put(-20,78){\makebox(10,10){$t_{n,k-\frac{1}{2}}^1$}}
\put(-20,88){\makebox(10,10){$t_{n,k+\frac{1}{2}}^1$}}
\put(165,65){\makebox(10,10){$t_{n-\frac{1}{2}}^2$}}
\put(165,115){\makebox(10,10){$t_{n+\frac{1}{2}}^2$}}
\put(165,40){\makebox(10,10){$t_{n-1}^2$}}
\put(163,90){\makebox(10,10){$t_n^2$}}

\put(-3,45){\vector(0,1){5}}
\put(-3,45){\vector(0,-1){5}}
\put(-15,40){\makebox(10,10){$\dt_1$}}
\put(185,45){\vector(0,1){25}}
\put(185,45){\vector(0,-1){25}}
\put(190,40){\makebox(10,10){$\dt_2$}}


\put(75,10){\makebox(10,10){$x_{I+\frac{1}{2}}$}}
\put(35,10){\makebox(10,10){$\O_1$}}
\put(115,10){\makebox(10,10){$\O_2$}}

\end{picture}

\caption{Time-space discretization.}
\label{discretisation}
\end{figure}

\subsection{Discretization}
In each subdomain, the equation is discretised using a classical cell centered finite volume implicit scheme:

\begin{subequations}
\label{eq:scheme1D}
\begin{align}
\frac{h_j}{\dt_1}(p^{n,k+1}_j- p^{n,k}_j)   - (u_{\jp12}^{n,k+1}-u_{\jm12}^{n,k+1}) = h_j f^{n,k+1}_j \mbox{ for } j \leq I  \label{eq:scheme1D_dom1}\\
\frac{h_j}{\dt_2} (p^{n+1}_j- p^{n}_j) - (u_{\jp12}^{n+1}-u_{\jm12}^{n+1}) = h_j f^{n+1}_j \mbox{ for } j > I  \label{eq:scheme1D_dom2}
\end{align}
\end{subequations}
where $p^{n}_j$ is an approximation of the unknown in the space-time cell $(x_\jm12,x_\jp12)\times (t_{n-\demi},t_{n+\demi})$, $p^{n,k}_j$ in the cell $(x_\jm12,x_\jp12)\times (t_{n,k-\demi},t_{n,k+\demi})$ (see Figures \ref{figure:maillage1D} and \ref{discretisation}), 
$$
f^{n+1}_j=\frac{1}{h_j\dt_1 }\int_{x_\jm12}^{x_\jp12}\int_{t_{n-\demi}}^{t_{n+\demi}} f(x,t)dxdt
$$
 and 
$$
f^{n,k+1}_j=\frac{1}{h_j\dt_2 }\int_{x_\jm12}^{x_\jp12}\int_{t_{n,k-\demi}}^{t_{n,k+\demi}} f(x,t)dxdt\,.
$$
  Except for the boundary nodes (which are handled classically and not precised here) and interface node $x_{I+\demi}$, the flux approximation $u_{\jp12}^s$   is given by:
\begin{align}
u_{\jp12}^s = \frac{p_{j+1}^s-p_{j}^s}{h_\jp12} \mbox{ for  $s=n$ if  $j>I$ or $s=n,k$ if $j<I$ }\label{eq:flux_intern}
\end{align} 
For the approximation of the fluxes on the interface $u_{I+\demi}^{n+1}$ and $u_{I+\demi}^{n,k}$, we 
consider the space-time domain decomposition framework and introduce $p_{I+\demi}^{n+1}$ and $p_{I+\demi}^{n,k}$
the unknown approximations on the interface ${x_{I+\demi}}$. The flux approximations are classically 
obtained as 
\begin{eqnarray}
 u_{I+\demi}^{n,k} = \frac{p^{n,k}_{I+\demi}-p^{n,k}_{I}}{0.5h_I} \label{eq:flux1D1}\\
 u_{I+\demi}^{n+1} = \frac{p^{n+1}_{I+1}-p^{n+1}_{I+\demi}}{0.5h_{I+1}} \label{eq:flux1D2}
\end{eqnarray}
Discretizations in the two domains are linked by interface conditions which enforce flux and unknown continuity on the interface ${x_{I+\demi}}\times(t_{n-\demi},t_{n+\demi})$. We consider the following two possibilities:

\begin{subequations}
\label{eq:master2}
\begin{align}
\dt_2  u_{I+\demi}^{n+1} = \sum_{k=1}^{\K} \dt_1 u_{I+\demi}^{n,k} \label{eq:master2:flux}\\
 p_{I+\demi}^{n,k} =  p_{I+\demi}^{n+1} \quad k=1\cdots\K \label{eq:master2:pres}
\end{align}
\end{subequations}
or: 
\begin{subequations}
\label{eq:master1}
\begin{align}
 u_{I+\demi}^{n,k} =  u_{I+\demi}^{n+1} \quad k=1\cdots\K \label{eq:master1:flux} \\
\dt_2  p_{I+\demi}^{n+1} = \sum_{k=1}^{\K} \dt_1 p_{I+\demi}^{n,k}   \label{eq:master1:pres}
\end{align}
\end{subequations}
These interface conditions can be rewritten in terms of the L2 orthogonal projections on sets of piecewise constant functions in time (see section \ref{sec:fvontheinterface}). Both sets of conditions ensure local conservation for the coarse time step.
 
Another way to couple the fine and coarse grid which is somewhat more natural, is to introduce the unknown approximations  $p_{I+\demi}^{n+1}$ and $p_{I+\demi}^{n,k}$ not on the interface but rather in the neighbouring cells $p_{I}^{n+1}$ and $p_{I+1}^{n,k}$. The flux approximations on the interface are then directly expressed as for an interior edge and instead of (\ref{eq:flux1D1}),(\ref{eq:flux1D2}), we have:
\begin{eqnarray}
 u_{I+\demi}^{n,k} = \frac{p^{n,k}_{I+1}-p^{n,k}_{I}}{h_{I+\demi}} \label{eq:flux1D1_overlap}\\
 u_{I+\demi}^{n+1} = \frac{p^{n+1}_{I+1}-p^{n+1}_{I}}{h_{I+\demi}} \label{eq:flux1D2_overlap}
\end{eqnarray}
and the  interface conditions become:
\begin{subequations}
\label{eq:master2_overlap}
\begin{align}
\dt_2  u_{I+\demi}^{n+1} = \sum_{k=1}^{\K} \dt_1 u_{I+\demi}^{n,k} \label{eq:master2:flux_overlap}\\
 p_{I+1}^{n,k} =  p_{I+1}^{n+1} \quad k=1\cdots\K \label{eq:master2:pres_overlap}
\end{align}
\end{subequations}
or: 
\begin{subequations}
\label{eq:master1_overlap}
\begin{align}
 u_{I+\demi}^{n,k} =  u_{I+\demi}^{n+1} \quad k=1\cdots\K \label{eq:master1:flux_overlap} \\
\dt_2  p_{I}^{n+1} = \sum_{k=1}^{\K} \dt_1 p_{I}^{n,k}   \label{eq:master1:pres_overlap}
\end{align}
\end{subequations}
We have thus four possible coupling schemes: (\ref{eq:flux1D1}),(\ref{eq:flux1D2}),(\ref{eq:master2})--(\ref{eq:flux1D1}),(\ref{eq:flux1D2}),(\ref{eq:master1})--(\ref{eq:flux1D1_overlap}),(\ref{eq:flux1D2_overlap}),(\ref{eq:master2_overlap})--(\ref{eq:flux1D1_overlap}),(\ref{eq:flux1D2_overlap}),(\ref{eq:master1_overlap}) which are analysed in the sequel.
Equations (\ref{eq:flux1D1_overlap}),(\ref{eq:flux1D2_overlap}),(\ref{eq:master2_overlap}) are the approximations proposed in  \cite{ELV90}.

\subsection{Solution method \label{section:SolutionMethod}}
To solve the system of  algebraic equations for the unknowns values of the approximate solution between two coarse time levels, which includes all intermediate local time levels, we propose a method which combines the attractive feature of predictor-corrector approaches with the accuracy of domain decomposition type iterative algorithms.
The method includes a predictor stage which corresponds to the computation of the solution on the coarse grid time and an iterative corrector stage  where refined and coarse grids unknowns are solved alternatively,  until interface conditions are satisfied, using a Schwarz multiplicative Dirichlet Neumann algorithm \cite{quarteroni}.
If we consider (\ref{eq:flux1D1_overlap}),(\ref{eq:flux1D2_overlap}),(\ref{eq:master2_overlap}) interface conditions, the algorithm is:
\begin{itemize}
\item Predictor stage: computation of an approximate solution at coarse time step on the whole grid:  $\tilde{p}^{n+1}_j$ for all j 
\begin{equation} 
\frac{h_j}{\dt_2} (\tilde{p}^{n+1}_j- p^{n}_j) - (\tilde{u}_{\jp12}^{n+1}-\tilde{u}_{\jm12}^{n+1}) = h_j f^{n+1}_j  \mbox{ for all } j \label{eq:predictor}  
\end{equation}
\item Corrector iterative stage if  interface conditions (\ref{eq:master2_overlap}) are used: \\
Solve alternatively the equations in domain $\Oun$ using (\ref{eq:master2:pres_overlap}) interface condition and in domain $\Odeux$ using (\ref{eq:master2:flux_overlap}), until both interface conditions are satisfied simultaneously.
\end{itemize}
If   (\ref{eq:flux1D1_overlap}),(\ref{eq:flux1D2_overlap}),(\ref{eq:master1_overlap}) interface conditions are used, the corrector iterative stage is :
\begin{itemize} 
\item Corrector iterative stage: \\
Solve alternatively the equations in domain $\Oun$ using (\ref{eq:master1:flux_overlap}) interface condition and in domain $\Odeux$ using (\ref{eq:master1:pres_overlap}),  until both interface conditions are satisfied simultaneously.
\end{itemize}
These algorithms can also be written for (\ref{eq:flux1D1}),(\ref{eq:flux1D2}),(\ref{eq:master2}) or (\ref{eq:flux1D1}),(\ref{eq:flux1D2}),(\ref{eq:master1}) interface conditions.
We can notice that it is not necessary to iterate  the corrector stage until convergence to obtain a conservative solution. It is sufficient to stop the process after the resolution of the domain where Neumann interface conditions are imposed. The algorithm proposed in \cite{DKH95} consists in the predictor stage (\ref{eq:predictor}) and in the first iteration of the corrective stage with (\ref{eq:master1:flux_overlap}) interface condition.
This solution method is not limited to linear discrete equations and can be extended to the resolution of the non linear equations that arise in petroleum recovery applications. Following the idea of \cite{DKH95}, the predictor stage would then use a linear approximation of the problem in the refined domain in order to avoid convergence problems of the Newton algorithm, while the iterative corrector stage would consider the non linear problem.


\section{Finite volume discretization}
\label{sec:continuouslevel}

Problem \eqref{pb_cont} is rewritten as a domain decomposition problem. The domain $\Omega$ is decomposed into two non overlapping subdomains $\Omega_1$ and $\Omega_2$ ($\bar{\Omega}_1\cup\bar{\Omega}_2=\bar{\Omega}$ and $\Omega_1\cap\Omega_2=\emptyset$). The interface is denoted by $\Gamma = \bar{\Omega}_1 \cap \bar{\Omega}_2$. Problem \eqref{pb_cont} is equivalent to: 

Find $p^1 : \Omega_1 \times [0,T] \mapsto \R$ and $p^2 : \Omega_2 \times [0,T] \mapsto \R$ such that:
\begin{subequations}
\label{eq:pb_cont_ddm}
\begin{eqnarray}
\dpdt{i}(x,t) - \Delta p^i(x,t) = f(x,t) \quad \forall x \in \Omega_i \quad \forall t \in [0,T] \quad \forall i \in \{1,2\} \label{eq_cont} \\
p^i(x,0) = p_0(x) \quad \forall x \in \Omega_i \quad \forall i \in \{1,2\} \\
p^i(x,t) = 0 \quad \forall x \in \partial \Omega_i \cap \partial \Omega \quad \forall t \in [0,T] \quad \forall i \in \{1,2\} \\
\dpdn{1}(x,t) + \dpdn{2}(x,t) = 0 \quad \forall x \in \Gamma \quad \forall t \in [0,T] \label{trans_cont} \\
p^1(x,t) = p^2(x,t) \quad \forall x \in \Gamma \quad \forall t \in [0,T] 
\end{eqnarray}
\end{subequations}
where $n_i$ is the outward normal to domain $\Omega_i$, $i=1,2$.

\label{sec:discretisation}
Problem \eqref{eq:pb_cont_ddm}  is discretized using a cell centered 
finite volume scheme in each subdomain \cite{EGH00}. We choose this scheme as an example
but other schemes would be possible as well.

\subsection{Mesh and definition}
\label{sec:mesh}
For $i =1,2$, let $ {\mathcal{T}}_i $ be a set of closed polygonal subsets associated with $ \Omega_i $ such that $\bar\Omega_i=\cup_{K \in \mathcal{T}_i}K$. We shall denote $h=\max_{i\in \{1,2\},K \in \mathcal{T}_i } diam(K)$ its mesh size. We shall use the following notation for all $ i =1,2$.
\begin{itemize}
\item $\mathcal{E}_{\Omega_i}$ is the set of faces of $\mathcal{T}_i$.
\item $\mathcal{E}_{iD}$ is the set of faces such that
$\partial\Omega_i \cap\partial\Omega= \cup_{\epsilon
\in \mathcal{E}_{iD}} \epsilon$ (let us recall that
a Dirichlet boundary condition will be imposed on $\partial\Omega_i
\cap \partial\Omega$). 
\item $\mathcal{E}_{i}$ is the set of faces such that
$\partial\Omega_i {\setminus}\partial\Omega= \cup_{\epsilon
\in \mathcal{E}_{i}} \epsilon$. Grids are matching on the interface so that 
$$
\mathcal{E}:=\mathcal{E}_{1}=\mathcal{E}_{2}\,.
$$
\item $\forall K \in \mathcal{T}_i$,
\subitem $\mathcal{E}(K)$ denotes the set of faces of $K$.
\subitem $\mathcal{E}_{iD}(K)=\mathcal{E}(K)\cap\mathcal{E}_{iD}$
is the set of faces of $K$ which are
on $\partial\Omega_i \cap \partial\Omega$.
\subitem $\mathcal{E}_{i}(K)=\mathcal{E}(K)\cap\mathcal{E}_{i}$
is the set of faces of $K$ which are
on $\partial\Omega_i {\setminus} \partial\Omega$. 
\subitem $\mathcal{N}_i(K)=\{ K' \in \mathcal{T}_i| K\cap K'
\in \mathcal{E}_{\Omega_i} \}$ is the set of the control
cells adjacent to $K$ in $\Omega_i$.
\subitem $K_i(\eps)$ denotes the cell of $\T_i$ adjacent to $\eps\in\E_i\cup\mathcal{E}_{iD}$.
\item the time step in subdomain $\O_i$ is denoted by $\dt_i$, and $N_i$ denotes the number of time steps of the simulation so that $N_1\dt_1=N_2\dt_2=T$. Parameters $\dt_i$, $N_i$ satisfy
\begin{equation}
\frac{N_1}{N_2} = \frac{\dt_2}{\dt_1} = \K \in \mathbb{N}^*
\end{equation}
\item Let $[0,T]_{\dt_i}$ denote the discretization of the time interval $[0,T]$ in each subdomain $\Omega_i$ : $[0,T]_{\dt_i}=(t_n^i)_{n=1,\ldots,N_i}$, with $t_n^i = (n-\demi) \dt_i$ ; since the time discretization in $\Omega_1$ is a refinement of that in $\Omega_2$, we shall also write  : $t_{n,k}^1 = (n-1) \dt_2 + (k-\demi) \dt_1$, $n=1,\ldots,N_2$, $k=1,\ldots,\K$
\end{itemize} 
 We make the following geometrical assumptions on the global mesh : $\mathcal{T} = \mathcal{T}_1 \cup \mathcal{T}_2$
\begin{assumption}\label{hyp_maillage}
\rm $\mathcal{T}$ is a finite volume admissible mesh, i.e., $\mathcal{T}$ is a set of closed subsets of dimension $d$ such that
\begin{itemize}
\item 
for any
$(K,K')\in \mathcal{T}^2 $ with $ K \neq K'$, one has either $[K\,K']:=K\cap K' \in \mathcal{E}_{\Omega_1} \cup \mathcal{E}_{\Omega_2}$ or ${\rm dim}(K\cap K' ) < d-1$
\item for $i=1,2$, there exist points $(y_\epsilon)_{\eps \in \mathcal{E}_{\Omega_i}}$ on the faces  and points $(x_K)_{K \in \mathcal{T}_i}$ inside the cells such that (see figure~\ref{figure:geomassump})
\begin{itemize} 
\item for any adjacent cells $K$ and $K'$, the straight line $[ x_K , x_{K'} ]$ is perpendicular to the face $[K\, K']$ and $[ x_K , x_{K'} ] \cap [K\, K']  =\{ y_\epsilon \}$
\item for any face $\epsilon \in \mathcal{E}_{iD}$, let $K(\epsilon)\in \toi$ be such that $\epsilon\subset K$: then the straight line  $[ x_{K(\epsilon)} , y_{\eps} ]$ is
perpendicular to $\epsilon$ 
\end{itemize}
\item Each mesh $\mathcal{T}_i$, $i=1,2$ has at least one interior cell. 
\end{itemize}
\end{assumption}
\begin{figure}
\begin{center}
\includegraphics[height=3cm]{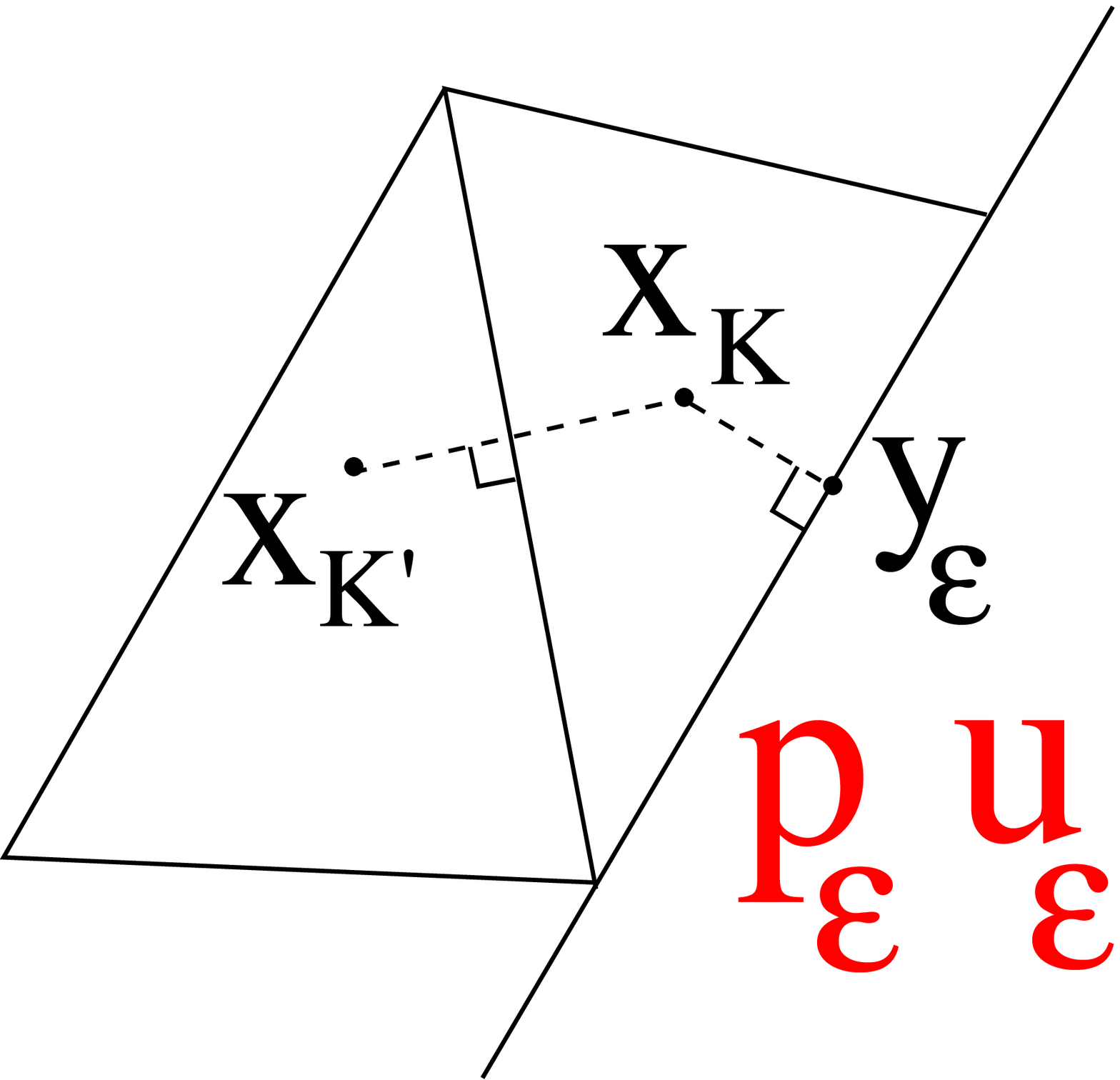}
\caption{Assumption \ref{hyp_maillage}}
    \label{figure:geomassump}
\end{center}
\end{figure}
\begin{notation}
 For all $\epsilon\in \mathcal{E}_{i}\cup\mathcal{E}_{iD}$, $i=1,2$, $d_\epsilon^i$ denotes the distance between $x_{K_i(\epsilon)}$ and $y_\epsilon$.   
\end{notation}

\subsection{Cell centered finite volume scheme in the subdomains}
\label{sec:fvsubdomains}
The unknowns of the scheme and what they aim to approximate are ($i=1,2$):
\begin{align*}
p_K^{i,n} & \simeq p^i(x_K,t_n),\ K\in\T_i \\
p_\eps^{i,n} & \simeq p^i(y_\eps,t_n),\  \eps\in\E_i\cup\E_{iD}\\
u_\eps^{i,n} & \simeq \dpdn{i}(y_\eps,t_n),\ \eps\in\E_i\cup\E_{iD}
\end{align*}
The numerical flux is defined by:
\begin{equation} \label{rel_p_u}
u_\eps^{i,n} = \frac{p_\eps^{i,n}-p_{K_i(\eps)}^{i,n}}{d_\eps^i} \quad \forall \eps \in \E_i \cup \E_{iD} \quad \forall n \in \{0..N_i-1\} \quad \forall i \in \{1,2\}
\end{equation}
The scheme is defined by (see e.g. \cite{EGH00} for its derivation):
\begin{multline} \label{eq_schema}
\frac{p_K^{i,n+1}-p_K^{i,n}}{\dt_i} \meas(K) - \sumKp{i} \frac{p_{K'}^{i,n+1}-p_K^{i,n+1}}{d(x_K,x_{K'})} \meas([KK']) \\
- \sumEDE{i} u_\eps^{i,n+1} \meas(\eps) = f_K^{i,n+1} \meas(K) \quad \forall K \in \T_i \quad \forall n \in 0..N_i-1
\end{multline}
where $d(x_K,x_{K'})$ is the distance between $x_K$ and $x_{K'}$ and $f_K^{i,n}$ is an approximation to $\frac{1}{\dt_i} \int_{t_{n-1/2}^i}^{t_{n+1/2}^i} \frac{1}{\meas(K)} \int_K f$. The initial and boundary conditions are discretized by:
\begin{eqnarray}
p_K^{i,0} = p_0(x_K) \quad \forall K \in \T_i \quad \forall i \in \{1,2\} \\
p_\eps^{i,n+1} = 0 \quad \forall \eps \in \E_{iD} \quad \forall n \in 0..N_i-1 \quad \forall i \in \{1,2\}
\end{eqnarray}
When there is no domain decomposition, this scheme has been analyzed in \cite{EGH00} in the more general case of discontinuous coefficients, and it is proven to be of order 1 for a discrete $H^1$-norm. 

In order to define the domain decomposition discretization scheme, we shall define in section~\ref{sec:fvontheinterface} the matching conditions for the diffusive fluxes.

{\bf Discrete spaces}\\
\begin{itemize}
\item 
\begin{equation}
  \begin{array}{l}
P_0(\T_i \times [0,T]_{\dt_i}) = \left\{ p : \bar\Omega_i\times [0,T]\mapsto \R \backslash\  \forall n \in \{0,\ldots,N_i-1\}\right. \\
\hfill  \forall K \in \T_i \ \  p_{|K \times (t_{n-1/2}^i,t_{n+1/2}^i)} \equiv C^{t}\\
\hfill \left. \mbox{ and } \forall \epsilon\in\E_i\cup\E_{iD}\ \,,  \ p_{|\epsilon \times (t_{n-1/2}^i,t_{n+1/2}^i)} \equiv C^{t}\right\}
  \end{array}
\end{equation}
\item Similarly, $P_0(\mathcal{E} \times [0,T]_{\dt_i}) $ is the space of piecewise constant functions on the interface for the time mesh of subdomain $\Omega_i$.
\item $P_0([0,T]_{\dt_i}) = \left\{ p : [0,T]\mapsto \R\backslash\ \ \forall n \in \{0,\ldots,N_i-1\}\ p_{|(t_{n-1/2}^i,t_{n+1/2}^i)} \equiv C^{t} \  \right\}$
\end{itemize}
These are spaces of piecewise constant functions. 

Let $p^i\in P_0(\T_i \times [0,T]_{\dt_i})$, we denote its restriction 
\begin{itemize}
\item to $\bar\Omega_i\times \{t^i_n \}$ by $p^{i,n}$, for $n\in \{0,\ldots,N_i-1\}$
\item to $\epsilon\times [0,T]_{\dt_i}$ by $p^i_\epsilon$ for $\epsilon\in\E$
\item to $\E\times [0,T]_{\dt_i}$ by $p^i_\E$ 
\end{itemize}
We introduce the following norms and semi-norms for any $p^i\in P_0(\T_i \times [0,T]_{\dt_i})$:
\begin{gather*}
\|p^{i,n}\|_{L^2(\O_i)}^2 = \sum_{K\in\T_i} (p_K^{i,n})^2 \meas(K) \\
\|p^i\|_{L^2(0,T;L^2(\O_i))}^2 = \sum_{n=0}^{N_i} \dt_i \|p^{i,n}\|_{L^2(\O_i)}^2 
\end{gather*}
and
\begin{equation}
  \label{eq:seminomrH1}
  \begin{array}{l}
\ds |p^{i,n}|_{1,\T_i}^2 = \sum_{K\in\T_i} \left( 
\ds\sumKp{i} \frac{(p_{K'}^{i,n}-p_K^{i,n})^2}{d(x_K,x_{K'})} \meas([KK'])\right.\hfill\\
\ds\hfill\left.+ \sum_{\eps\in \E_i(K)} \frac{(p_\eps^{i,n}-p_K^{i,n})^2}{d(x_K,y_\eps)} \meas(\eps) + \sum_{\eps\in \E_{iD}(K)} \frac{(p_K^{i,n})^2}{d(x_K,y_\eps)} \meas(\eps) 
\right)   
  \end{array}
\end{equation}
and
\begin{equation}
  \label{eq:seminomrL2TH1}
|p^i|_{1,\T_i,\dt_i}^2 = \sum_{n=0}^{N_i} \dt_i |p^{i,n}|_{1,\T_i}^2
\end{equation}
\begin{definition}
Let $p^i,u^i \in P_0(\toi\times [0,T]_{\dt_i})$, $i=1,2$, we define a discrete scalar product by:
\[
\sumi \lprod u^i,p^i \rprod := 
   \sumi \sumn{i} \dt_i \sumeps u_\eps^{i,n+1} p_\eps^{i,n+1} \meas(\eps)
\]
\end{definition}

\begin{notation}
Let $i=1,2$, $p^i\in  P_0(\T_i \times [0,T]_{\dt_i})$, for $\epsilon\in\E$, $u^i_\epsilon(p)$ denotes the associated numerical flux defined by \eqref{rel_p_u}. Very often, we will simply write $u^i_\epsilon$ and $u^i_\E=(u_\eps(p^i))_{\eps\in\E_i}$.   
\end{notation}

\subsection{Finite volume on the interface}
\label{sec:fvontheinterface}

In order to enforce the weak continuity of the primary unknown $p$ and of its normal derivative (denoted by $u$) across the interface $\Gamma\times [0,T]$, we introduce $\mathcal{Q}_i$ the $L^2$ orthogonal projection onto $P_0([0,T]_{\dt_i})$. We have the following compatibility condition:
\begin{lemma}\label{hyp_transposition}
\rm For all  $u_i \in P_0([0,T]_{\dt_i})$, $i=1,2$,
\[
\langle \mathcal{Q}_1(u_2),u_1\rangle_{L^2( [0,T])}=\langle u_2,\mathcal{Q}_2(u_1)\rangle_{L^2([0,T])}
\]
\end{lemma}
As in mortar methods \cite{Bernardi:1994:NNA}, we consider that one subdomain enforces the weak continuity of the primary unknown which is interpreted as the Dirichlet interface condition. This subdomain is called the master. The other subdomain is called the slave. It enforces the weak continuity of the normal derivative which corresponds to a Neumann interface condition. 

Since here the interfaces are non matching only in the time direction, it is possible to define matching conditions locally on each interface face $\eps\in\E$.\\

{\bf Interface Scheme \eqref{cond_trans} based on interface unknowns}
The subscript $m$ will denote the master subdomain and $e$ the slave ($\{m,e\} = \{1,2\}$), the interface conditions on $\Gamma\times [0,T]$ read:
\begin{equation}
\tag{IS$_1$}
\label{cond_trans}
\left\{
\begin{aligned}
u_{\eps}^m =& \mathcal{Q}_m(-u_{\eps}^e) \\
p_{\eps}^e =& \mathcal{Q}_e(p_{\eps}^m)
\end{aligned}
\right.
\quad \forall \eps \in \E
\end{equation}

{\bf Overlapping interface scheme \eqref{cond_trans_rec}}
The Dirichlet boundary condition is modified but not the Neumann one:
\begin{equation}
\tag{IS$_2$} 
\label{cond_trans_rec}
\left\{
\begin{aligned}
u_{\eps}^m =& \mathcal{Q}_m(-u_{\eps}^e) \\
p_{\eps}^e + d_\eps^m u_{\eps}^e =& \Q_e(p_{\eps}^m - d_\eps^m u_{\eps}^m)
\end{aligned}
\right.
\quad \forall \eps \in \E
\end{equation}
The modified Dirichlet interface condition comes from the following relations:
\begin{equation*}
\begin{split}
u_{\eps}^e =& \frac{\Q_e(p_{K_m(\eps)}^m) - p_{K_e(\eps)}^e}{d_\eps^m+d_\eps^e} \\
=& \frac{\Q_e(p_{\eps}^m - d_\eps^m u_{\eps}^m) - p_{\eps}^e + d_\eps^e u_{\eps}^e}{d_\eps^m+d_\eps^e}
\end{split}
\end{equation*}
where $d_\eps^i = d(x_{K_i(\eps)},y_\eps)$. The first line of the above formula is somewhat natural. When writing the finite volume scheme for a cell $K_e(\epsilon)$ adjacent to the interface in the ``slave'' subdomain, it is necessary to approximate the flux on the face $\epsilon$. This is done using pressure values from both sides of the interface: the pressure in the ``slave'' subdomain and pressures values in the neighboring ``master'' subdomain. These last values are made compatible with the ``slave'' unknowns by using the transmission operator $\mathcal{Q}_e$. Finally, all quantities are expressed in terms of interface values in order to ease a comparison with \eqref{cond_trans}. 

Due to the fact that the large time step $\delta t_2$ is a multiple of the small time step $\delta t_1$, we have a simple form for the $L^2$ projection operators.
\begin{lemma}
We have $\mathcal{Q}_2 : P_0( [0,T]_{\dt_1}) \mapsto P_0( [0,T]_{\dt_2})$ and $\mathcal{Q}_1 : P_0( [0,T]_{\dt_2}) \mapsto P_0([0,T]_{\dt_1})$ 
\begin{equation} \label{schema_constant}
\left\{
\begin{aligned}
\mbox{For }v_2\in P_0( [0,T]_{\dt_2}),\ 
\Q_1(v_2)_{|(t^1_{n-1/2},t^1_{n+1/2})} = v_{2|(t^1_{n-1/2},t^1_{n+1/2})} \quad \forall\, n \in \{0,\ldots,N_1-1\} \\
\mbox{For }v_1\in P_0( [0,T]_{\dt_1}),\ 
\Q_2(v_1))_{|(t^2_{n-1/2},t^2_{n+1/2})} = \frac{1}{\dt_2} \int_{t^2_{n-1/2}}^{t^2_{n+1/2}} v_1 \quad \forall\, n \in \{0,\ldots,N_2-1\}
\end{aligned}
\right.
\end{equation}
\end{lemma}

We also need the following technical assumptions. For the family of meshes we consider, the mesh close to the interface is not too stretched:

\begin{assumption} \label{hyp_alpha}
There exists a constant $\alpha > 0$ such that $d_\eps^m \leqslant \alpha d_\eps^e$ for all $\eps \in \E$.
\end{assumption}
\noindent We also need a geometric assumption:
\begin{assumption} \label{hyp_yepsilon}
For all $\eps \in \E$, $y_\eps$ is the  barycenter of the face $\eps$ and for $i=1,2$, 
\[
\frac{diam(\eps)^2}{d(x_{K_i(\eps)},y_\eps)} = O(h)
\]
\end{assumption}

\section{A general stability result}
 \label{sec:stability}

We prove a stability result for \eqref{eq_schema} modified by the introduction of discretization error terms  $R_K^{i,n}$, $R_{KK'}^{i,n}$ and $R_\eps^{i,n}$ as well as  $F_K^{i,n}$ which will be defined precisely in the sequel ($i\in\{1,2\}$):
\begin{multline} \label{eq_schema_gen}
\left( \frac{p_K^{i,n+1}-p_K^{i,n}}{\dt_i} - R_K^{i,n+1} \right) \meas(K) - \sumKp{i} \left( \frac{p_{K'}^{i,n+1}-p_K^{i,n+1}}{d(x_K,x_{K'})} - R_{KK'}^{i,n+1} \right) \meas([KK']) \\
- \sumEDE{i} (u_\eps^{i,n+1} - R_\eps^{i,n+1}) \meas(\eps) = F_K^{i,n+1} \meas(K) \quad \forall K \in \T_i \quad \forall n \in \{0,\ldots,N_i-1\}\,.
\end{multline}
Formula \eqref{rel_p_u}  is modified as well by introducing error terms at each time step $n$
\begin{equation} \label{def_erreur_consistance}
U_\eps^{i,n} = u_\eps^{i,n} - \frac{p_\eps^{i,n}-p_{K_i(\eps)}^{i,n}}{d_\eps^i} \quad \forall \eps \in \E_i \cup \E_{iD} \quad \forall n \in \{1,\ldots,N_i\} \quad \forall i \in \{1,2\}
\end{equation}
(with $p_\eps^{i,n} = 0$ for all $\eps \in \E_{iD}$). 

The interface conditions \eqref{cond_trans} are modified in the following manner by the terms $\P_{\eps} \in P_0(\E \times [0,T]_{\dt_e})$ and $\U_{\eps} \in P_0(\E \times [0,T]_{\dt_m})$ which will be defined in the sequel:
%
\begin{equation} \label{def_erreur_transmission}
\tag{IS$_1$'}
\left\{
\begin{aligned}
p_{\eps}^e =& \mathcal{Q}_e(p_{\eps}^m) + \P_{\eps}^e \\
u_{\eps}^m =& \mathcal{Q}_m(-u_{\eps}^e) + \U_{\eps}^m 
\end{aligned}
\right. \quad \forall \eps \in \E
\end{equation}
Interface conditions\eqref{cond_trans_rec} are similarly modified for all $\eps\in\E$ :
\begin{equation*}
\begin{split}
u_{\eps}^e =& \frac{\Q_e(p_{K_m(\eps)}^m) + \P_{\eps} - p_{K_e(\eps)}^e}{d_\eps^m+d_\eps^e} \\
=& \frac{\Q_e(p_{\eps}^m - d_\eps^m u_{\eps}^m + d_\eps^m U_{\eps}^m) + \P_{\eps} - p_{\eps}^e + d_\eps^e u_{\eps}^e - d_\eps^e U_{\eps}^e}{d_\eps^m+d_\eps^e}
\end{split}
\end{equation*}
The interface conditions \eqref{cond_trans_rec} are thus modified in the following manner:
\begin{equation} \label{def_erreur_transmission_rec}
\tag{IS$_2$'}
\left\{
\begin{aligned}
u^m_{\eps} =& \Q_m(-u^e_{\eps}) + \U_{\eps} \\
p_{\eps}^e + d_\eps^m u_{\eps}^e =& \Q_e(p_{\eps}^m - d_\eps^m u_{\eps}^m + d_\eps^m U_{\eps}^m) + \P_{\eps} - d_\eps^e U_{\eps}^e
\end{aligned}
\right. \quad \forall \eps \in \E
\end{equation}


We make some additional assumptions. 
\begin{assumption} \label{hyp_RKK}
$R_{KK'} + R_{K'K} = 0 \quad \forall K \in \T_i \quad \forall K' \in \N_i(K) \quad \forall i \in \{1,2\}$
\end{assumption}

\begin{assumption} \label{hyp_R}
We suppose that the following estimates are satisfied by the consistency errors:\\
for all $i \in \{1,2\}$ and all $n \in \{1,\ldots,N_i\}$, we have:
\begin{align*}
U_\eps^{i,n} =& O(h^{\eta_1}+\dt_i^{\gamma_1}) \quad \forall K \in \T_i \quad \forall \eps \in \E_{i}(K)\cup\E_{iD}(K) \\
R_K^{i,n} =& O(h^{\eta_2}+\dt_i^{\gamma_2}) \quad \forall K \in \T_i \\
R_{KK'}^{i,n} =& O(h^{\eta_3}+\dt_i^{\gamma_3}) \quad \forall K \in \T_i \quad \forall K' \in \mathcal{N}_i(K) \\
R_\eps^{i,n} =& O(h^{\eta_4}+\dt_i^{\gamma_4}) \quad \forall K \in \T_i \quad \forall \eps \in \E_i(K)\cup\E_{iD}(K) \\
F_K^{i,n} =& O(h^{\eta_5}+\dt_i^{\gamma_5}) \quad \forall K \in \T_i\\
\P^{i,n}_{\eps} =& O(h^{\eta_6}+\dt_2^{\gamma_6})\quad \forall \eps\in\E \\
\U^{i,n}_{\eps} =& O(h^{\eta_7}+\dt_1^{\gamma_7})\quad  \forall \eps\in\E
\end{align*}
\end{assumption}

We first prove an estimate for \eqref{eq_schema_gen}-\eqref{def_erreur_consistance} not taking into account the interface conditions:
\begin{lemma} \label{lemme_sans_int}
Let $(p^i,u^i(p^i))$ satisfy \eqref{eq_schema_gen}-\eqref{def_erreur_consistance} and suppose that assumptions \ref{hyp_maillage} and \ref{hyp_RKK} are satisfied.\\
 Then :
\begin{multline} \label{eq_gen}
\demi \sumi \sumK{i} (p_K^{i,N_i})^2 \meas(K) 
+ \demi \sumi \sumn{i} \dt_i \sumK{i}\ \sumKp{i} \frac{(p_{K'}^{i,n+1}-p_K^{i,n+1})^2}{d(x_K,x_{K'})} \meas([KK']) \\
+ \sumi \sumn{i} \dt_i \sumK{i}\ \sumEDE{i} \frac{(p_\eps^{i,n+1}-p_K^{i,n+1})^2}{\dxy} \meas(\eps) \\
\leqslant  
\sumi \lprod u^i,p^i \rprod
+ \demi \sumi \sumK{i} (p_K^{i,0})^2 \meas(K) \\
+ \sumi \sumn{i} \dt_i \sumK{i} (F_K^{i,n+1}+R_K^{i,n+1}) p_K^{i,n+1} \meas(K) \\
+ \demi \sumi \sumn{i} \dt_i \sumK{i}\ \sumKp{i} R_{KK'}^{i,n+1} (p_{K'}^{i,n+1}- p_K^{i,n+1}) \meas([KK']) \\
- \sumi \sumn{i} \dt_i \sumK{i}\ \sumEDE{i} R_\eps^{i,n+1} p_K^{i,n+1} \meas(\eps) \\
- \sumi \sumn{i} \dt_i \sumK{i}\ \sumEDE{i} U_\eps^{i,n+1} (p_\eps^{i,n+1}-p_K^{i,n+1}) \meas(\eps)
\end{multline}
\end{lemma}

\begin{proof}
In each subdomain $\Omega_i$, we multiply \eqref{eq_schema_gen} by $\dt_i p_K^{i,n+1}$ and we sum over cells $K$ and time step $n$ and make use of the following formula:
\begin{gather*}
\begin{aligned}
\sumn{i} (p_K^{i,n+1}-p_K^{i,n}) p_K^{i,n+1} =& \demi \sumn{i} \left[ (p_K^{i,n+1})^2 - (p_K^{i,n})^2 + (p_K^{i,n+1}-p_K^{i,n})^2 \right] \\
=& \demi \left[ (p_K^{i,N_i})^2 - (p_K^{i,0})^2 + \sumn{i} (p_K^{i,n+1}-p_K^{i,n})^2 \right] \\
\geqslant& \demi \left[ (p_K^{i,N_i})^2 - (p_K^{i,0})^2 \right]
\end{aligned} \\
\sumK{i} \sumKp{i} \frac{p_{K'}^{i,n}-p_K^{i,n}}{d(x_K,x_{K'})} p_K^{i,n} \meas([KK']) = - \demi \sumK{i} \sumKp{i} \frac{(p_{K'}^{i,n}-p_K^{i,n})^2}{d(x_K,x_{K'})} \meas([KK']) \\
\intertext{Using the fact that $R_{KK'} + R_{K'K} = 0$ (Assumption \ref{hyp_RKK}), we have :}
\sumK{i} \sumKp{i} R_{KK'}^{i,n} p_K^{i,n} \meas([KK']) = - \demi \sumK{i} \sumKp{i} R_{KK'}^{i,n} (p_{K'}^{i,n} - p_K^{i,n}) \meas([KK'])
\end{gather*}
Then, we get:
\begin{multline*}
\demi \sumK{i} \left( (p_K^{i,N_i})^2 - (p_K^{i,0})^2 \right) \meas(K) 
+ \demi \sumn{i} \dt_i \sumK{i}\ \sumKp{i} \frac{(p_{K'}^{i,n+1}-p_K^{i,n+1})^2}{d(x_K,x_{K'})} \meas([KK']) \\
- \sumn{i} \dt_i \sumK{i}\ \sumEDE{i} u_\eps^{i,n+1} p_K^{i,n+1} \meas(\eps)
\leqslant \sumn{i} \dt_i \sumK{i} F_K^{i,n+1} p_K^{i,n+1} \meas(K) \\
+ \sumn{i} \dt_i \sumK{i} R_K^{i,n+1} p_K^{i,n+1} \meas(K) + \demi \sumn{i} \dt_i \sumK{i}\ \sumKp{i} R_{KK'}^{i,n+1} (p_{K'}^{i,n+1}- p_K^{i,n+1}) \meas([KK']) \\
- \sumn{i} \dt_i \sumK{i}\ \sumEDE{i} R_\eps^{i,n+1} p_K^{i,n+1} \meas(\eps)
\end{multline*}
We sum over the subdomains and make use of the following formula for any $\eps \in \E_i \cup \E_{iD}$ and $K=K(\eps)$,
\begin{equation*}
\begin{split}
- u_\eps^{i,n} p_K^{i,n} =& - u_\eps^{i,n} p_\eps^{i,n} - u_\eps^{i,n} (p_K^{i,n}-p_\eps^{i,n}) \\
=& - u_\eps^{i,n} p_\eps^{i,n} + \left(\frac{p_K^{i,n}-p_\eps^{i,n}}{d(x_K,y_\eps)} - U_\eps^{i,n} \right) (p_K^{i,n}-p_\eps^{i,n}) \\
=& - u_\eps^{i,n} p_\eps^{i,n} + \frac{(p_K^{i,n}-p_\eps^{i,n})^2}{d(x_K,y_\eps)} - (p_K^{i,n}-p_\eps^{i,n}) U_\eps^{i,n}
\end{split}
\end{equation*} 
Then, \eqref{eq_gen} follows from $p_\eps^{i,n} = 0$ for all $\eps \in \E_{iD}$.
\end{proof}

We focus now on the interface terms on $\Gamma$ that appear in the first term of the right hand side of \eqref{eq_gen}. We first consider the interface matching conditions  \eqref{def_erreur_transmission}.

\begin{lemma} \label{lemme}
Let $(p,u)$ satisfy \eqref{eq_schema_gen}, \eqref{def_erreur_consistance}, \eqref{def_erreur_transmission}. 
Then,
we have:
\begin{equation} \label{estim}
\sumi \lprod u^i,p^i \rprod = \lprod p^m,\U_{\E} \rprod + \lprod u^e,\P_{\E} \rprod
\end{equation}
\end{lemma}

\begin{proof}
From the interface conditions \eqref{def_erreur_transmission}, we have:
\begin{equation*} \label{term_int_sans_rec}
\begin{split}
\sumn{m}& \dt_m \sumeps u_\eps^{m,n+1} p_\eps^{m,n+1} \meas(\eps) + \sumn{e} \dt_e \sumeps u_\eps^{e,n+1} p_\eps^{e,n+1} \meas(\eps) \\ 
=& \lprod u^m,p^m \rprod + \lprod u^e,p^e \rprod \\
=& \lprod \Q_m(-u^e)+\U_{\E},p^m \rprod + \lprod u^e,\Q_e(p^m)+\P_{\E} \rprod \\
=& \underbrace{\lprod -\Q_m(u^e),p^m \rprod + \lprod  u^e,\Q_e(p^m) \rprod}_{=0 \text{ (Lemma \ref{hyp_transposition})}} \\
&+ \lprod \U_{\E},p^m \rprod + \lprod u^e,\P_{\E} \rprod
\end{split}
\end{equation*}
\end{proof}

We consider now the interface scheme \eqref{def_erreur_transmission_rec}

\begin{lemma} \label{lemme_rec}
Let $(p,u)$ satisfy \eqref{eq_schema_gen}, \eqref{def_erreur_consistance}, \eqref{def_erreur_transmission_rec}. 
Then, we have
\begin{multline} \label{estim_rec}
\sumi \lprod p^i,u^i \rprod \leqslant \lprod p^m,\U_{\E} \rprod 
+ \lprod u^e , \P_{\E} - d_\E^e U_{\E}^e - d_\E^m (\U_{\E} - U_{\E}^m) \rprod
\end{multline}
where $d_\E^i$ is the piecewise constant function on $\Gamma$ such that $d_\E^i(x) = d_\eps^i$ for all $x \in \eps \in \E$.
\end{lemma}

\begin{proof}
Using the interface conditions \eqref{def_erreur_transmission_rec}, we have: 
\[
\begin{array}{l}
\sumi \lprod p^i,u^i \rprod = \lprod  -\Q_m(u^e) + \U_\E , p^m \rprod \\
   + \lprod u^e , \Q_e(p^m - d_\E^m u_{\E}^m + d_\E^m U_{\E}^m) + \P_{\E} - d_\E^m u_{\E}^e - d_\E^e U_{\E}^e \rprod
\end{array}
\]
Since we use $L^2$ projection, we have
\[
\begin{array}{l}
\sumi \lprod p^i,u^i \rprod = \lprod  - u^e + \U_\E , p^m \rprod \\
   + \lprod u^e , p^m - d_\E^m u_{\E}^m + d_\E^m U_{\E}^m + \P_{\E} - d_\E^m u_{\E}^e - d_\E^e U_{\E}^e \rprod
\end{array}
\]
Simplifying the relation and using again the first equation of  \eqref{def_erreur_transmission_rec}, 
\[
\begin{array}{l}
\sumi \lprod p^i,u^i \rprod = \lprod   \U_\E , p^m \rprod \\
   + \lprod u^e , - d_\E^m \Q_m(-u^e)-d_\E^m \U_{\E} + d_\E^m U_{\E}^m + \P_{\E} - d_\E^m u_{\E}^e - d_\E^e U_{\E}^e \rprod
\end{array}
\]
Since $\Q_m$ is a $L^2$ projection, we have
\[
\begin{array}{l}
\sumi \lprod p^i,u^i \rprod \leq \lprod   \U_\E , p^m \rprod \\
   + \lprod u^e , -d_\E^m \U_{\E} + d_\E^m U_{\E}^m + \P_{\E} - d_\E^e U_{\E}^e \rprod
\end{array}
\]

\end{proof}

\begin{theorem} \label{theorem}
 Suppose assumptions  \ref{hyp_maillage}, \ref{hyp_RKK} and \ref{hyp_R} hold. Let $(p,u)$ satisfy \eqref{eq_schema_gen}, \eqref{def_erreur_consistance}. If one of the two conditions is satisfied, 
\begin{enumerate}[i)]
\item  $(p,u)$ satisfies transmission conditions \eqref{def_erreur_transmission},
\item $(p,u)$ satisfies transmission conditions \eqref{def_erreur_transmission_rec} and assumption \ref{hyp_alpha} holds
\end{enumerate}
Then, we have the following estimate:
\begin{equation*}
\sumi |p^i|_{1,\T_i,\dt_i}^2 + 2 \sumi \|p^{i,N_i}\|_{L^2(\O_i)}^2 \leqslant 2 \sumi \|p^{i,0}\|_{L^2(\O_i)}^2 + O(h)^\eta + O(\dt_2)^\gamma
\end{equation*}
where $\eta = 2\min(\eta_j)_{j=1,\ldots,7}$ et $\gamma = 2\min(\gamma_j)_{j=1,\ldots,7}$.
\end{theorem}

\begin{proof}
The proof consists in estimating the terms in formula \eqref{eq_gen} of lemma~\ref{lemme_sans_int}. We often use the relation
\begin{equation*}
|ab| \leqslant \frac{C}{2} a^2 + \frac{1}{2C} b^2 \quad \forall (a,b) \in \mathbb{R}^2 \quad \forall C \in \mathbb{R}^*_+
\end{equation*}
with various constants $C_i$ which are independent of the parameters of the mesh size. For the terms that are classical in the finite volume theory, we simply write the estimate.

We begin with the estimate of the classical term 
\begin{gather*}
\begin{aligned}
\sumi \sumn{i} \dt_i \sumK{i}& F_K^{i,n+1} p_K^{i,n+1} \meas(K) 
\leqslant& \frac{C_5}{2} \sumi \sumn{i} \dt_i \sumK{i} (p_K^{i,n+1})^2 \meas(K) + O(h^{2\eta_5}+\dt_i^{2\gamma_5})
\end{aligned}
\end{gather*}
We also have the term
\[
\sumi \sumn{i} \dt_i \sumK{i} R_K^{i,n+1} p_K^{i,n+1} \meas(K) \leqslant \frac{C_2}{2} \sumi \sumn{i} \dt_i \sumK{i} (p_K^{i,n+1})^2 \meas(K) + O(h^{2\eta_2}+\dt_i^{2\gamma_2}) 
\]
and the term
\begin{gather*}
\begin{aligned}
- \sumi \sumn{i} \dt_i \sumK{i} \sumE{i} R_\eps^{i,n+1} p_\eps^{i,n+1} \meas(\eps) \leqslant& \frac{C_4}{2} \sumi \sumn{i} \dt_i \sumK{i} \sumE{i} (p_\eps^{i,n+1})^2 \meas(\eps) \\
&+ O(h^{2\eta_4}+\dt_i^{2\gamma_4})
\end{aligned}
\end{gather*}
We consider now 
\begin{gather*}
\begin{aligned}
- \sumi \sumn{i} \dt_i \sumK{i}& \sumEDE{i} (p_\eps^{i,n+1}-p_K^{i,n+1}) U_\eps^{i,n+1} \meas(\eps) \\
=& - \sumi \sumn{i} \dt_i \sumK{i}\ \sumEDE{i} \frac{p_\eps^{i,n+1}-p_K^{i,n+1}}{\sqrt{\dxy}} U_\eps^{i,n+1} \sqrt{\dxy} \meas(\eps) \\
\leqslant& \frac{C_1}{2} \sumi \sumn{i} \dt_i \sumK{i}\ \sumEDE{i} \frac{(p_\eps^{i,n+1}-p_K^{i,n+1})^2}{\dxy} \meas(\eps) \\
&+ \frac{1}{2C_1} \sumi \sumn{i} \dt_i \sumK{i}\ \sumEDE{i} (U_\eps^{i,n+1})^2 \dxy \meas(\eps) \\
\leqslant& \frac{C_1}{2} \sumi \sumn{i} \dt_i \sumK{i}\ \sumEDE{i} \frac{(p_\eps^{i,n+1}-p_K^{i,n+1})^2}{\dxy} \meas(\eps) \\
&+ \frac{1}{2C_1} O(h^{2\eta_1}+\dt_i^{2\gamma_1}) \sumi \sumn{i} \dt_i \sumK{i}\ \sumEDE{i} \dxy \meas(\eps) \\
\leqslant& \frac{C_1}{2} \sumi \sumn{i} \dt_i \sumK{i}\ \sumEDE{i} \frac{(p_\eps^{i,n+1}-p_K^{i,n+1})^2}{\dxy} \meas(\eps) + O(h^{2\eta_1}+\dt_i^{2\gamma_1})
\end{aligned} 
\end{gather*}
where we have used the following formula (see \cite{EGH00})
\begin{equation}
  \label{eq:vraivolume}
\sum_{K\in \mathcal{T}_i} \sum_{K' \in \mathcal{N}_i(K)} d(x_K,x_{K'})\, meas([K,K']
+  \sumK{i}\ \sumEDE{i} \dxy \,\meas(\eps) \\
\le d\, meas(\Omega_i) \nonumber
\end{equation}

In a classical way, we get 
\begin{gather*}
\begin{aligned}
\demi \sumi \sumn{i} \dt_i& \sumK{i}\ \sumKp{i} R_{KK'}^{i,n+1} (p_K^{i,n+1}-p_{K'}^{i,n+1}) \meas([KK']) \\
\leqslant& \frac{C_3}{4} \sumi \sumn{i} \dt_i \sumK{i}\ \sumKp{i} \frac{(p_K^{i,n+1}-p_{K'}^{i,n+1})^2}{d(x_K,x_{K'})} \meas([KK']) + O(h^{2\eta_3}+\dt_i^{2\gamma_3})
\end{aligned}
\end{gather*}
and
\begin{gather*}
\begin{aligned}
\sumi \sumn{i} \dt_i \sumK{i}& \sumEDE{i} R_\eps^{i,n+1} (p_\eps^{i,n+1} - p_K^{i,n+1}) \meas(\eps) \\
\leqslant& \frac{C_4'}{2} \sumi \sumn{i} \dt_i \sumK{i}\ \sumEDE{i} \frac{(p_\eps^{i,n+1} - p_K^{i,n+1})^2}{\dxy} \meas(\eps) + O(h^{2\eta_4}+\dt_i^{2\gamma_4})
\end{aligned}
\end{gather*}
We now focus on the interface terms in \eqref{eq_gen} starting with transmission scheme \eqref{def_erreur_transmission}.  By \eqref{estim} of lemma~\ref{lemme} :
\begin{equation*}
\begin{aligned}
\sumi \lprod u^i,p^i \rprod =& \lprod p^m,\U_{\E} \rprod + \lprod U_{\E}^e , \P_{\E} \rprod \\
&+ \sumn{e} \dt_e \sumeps \frac{p_\eps^{e,n+1}-p_K^{e,n+1}}{\dxy} \P_\eps^{n+1} \meas(\eps) \\
\leqslant& \frac{C_5}{2} \sumn{m} \dt_m \sumeps (p_\eps^{m,n+1})^2 \meas(\eps) 
+ \frac{C_6}{2} \sumn{e} \dt_e \sumeps \frac{(p_\eps^{e,n+1}-p_K^{e,n+1})^2}{\dxy} \meas(\eps) \\
&+ O(h)^{\min(\eta_1+\eta_6,2\eta_6,2\eta_7)} + O(\dt_2)^{\min(\gamma_1+\gamma_6,2\gamma_6,2\gamma_7)}
\end{aligned}
\end{equation*}
The analysis of the interface term in \eqref{eq_gen} with the transmission scheme \eqref{def_erreur_transmission_rec} is more involved. By \eqref{estim_rec} of lemma~\ref{lemme_rec}, we have the following additional term:
\begin{multline*}
-\lprod u^e, d_\E^e U_{\E}^e + d_\E^m (\U_{\E}-U_{\E}^m) \rprod = 
\underbrace{-\sumn{e} \dt_e \sumE{e} d_\eps^e (U_\eps^{e,n})^2 \meas(\eps)}_{=O(h^{2\eta_1}+\dt_e^{2\gamma_1})} \\
\underbrace{-\sumn{e} \dt_e \sumE{e} d_\eps^e \frac{p_\eps^{e,n}-p_{K_e(\eps)}^{e,n}}{d_\eps^e} U_\eps^{e,n} \meas(\eps)}_{\leqslant \frac{C_1'}{2}\sumn{e} \dt_e \sumE{e} \frac{(p_\eps^{e,n}-p_{K_e(\eps)}^{e,n})^2}{d_\eps^e} \meas(\eps) + O(h^{2\eta_1}+\dt_e^{2\gamma_1})}
\underbrace{-\sumn{e} \dt_e \sumE{e} d_\eps^m U_\eps^{e,n} (\U_{\eps}^n-U_{\eps}^{m,n}) \meas(\eps)}_{=O(h^{2\min(\eta_1,\eta_7)}+\dt_m^{2\min(\gamma_1,\gamma_7)})} \\
\underbrace{-\sumn{e} \dt_e \sumE{e} d_\eps^m \frac{p_\eps^{e,n}-p_{K_e(\eps)}^{e,n}}{d_\eps^e} (\U_{\eps}^n-U_{\eps}^{m,n}) \meas(\eps)}_{\leqslant \frac{C_2'}{2}\sumn{e} \dt_e \sumE{e} \frac{(p_\eps^{e,n}-p_{K_e(\eps)}^{e,n})^2}{d_\eps^e} \meas(\eps) + O(h^{2\min(\eta_1,\eta_7)}+\dt_m^{2\min(\gamma_1,\gamma_7)}) \sumn{e} \dt_e \sumE{e} \frac{(d_\eps^m)^2}{d_\eps^e} \meas(\eps)}
\end{multline*}
By assumption~\ref{hyp_alpha}, the last term is under control since $\sumn{e} \dt_e \sumE{e} \frac{(d_\eps^m)^2}{d_\eps^e} \meas(\eps) \leqslant \alpha T \meas(\O_m)$. \\

Summing up all these estimates, we have:
\begin{multline*}
\demi \sumi \sumK{i} (p_K^{i,N_i})^2 \meas(K) + \left( \demi - \frac{C_3}{4} \right) \sumi \sumn{i} \dt_i \sumK{i}\ \sumKp{i} \frac{(p_{K'}^{i,n+1}-p_K^{i,n+1})^2}{d(x_K,x_{K'})} \meas([KK']) \\
 + \left( 1 - \frac{C_1}{2} - \frac{C_1'}{2} - \frac{C_2'}{2} - \frac{C_4'}{2} - \frac{C_6}{2} \right) \sumi \sumn{i} \dt_i \sumK{i}\ \sumEDE{i} \frac{(p_\eps^{i,n+1}-p_K^{i,n+1})^2}{\dxy} \meas(\eps) \\
 \\
\leqslant \demi \sumi \sumK{i} (p_K^{i,0})^2 \meas(K) + \frac{C_2+C_5}{2} \sumi \sumn{i} \dt_i \sumK{i} (p_K^{i,n+1})^2 \meas(K) \\ 
+ \frac{C_4+C_7}{2} \sumi \sumn{i} \dt_i \sumK{i}\ \sumE{i} (p_\eps^{i,n+1})^2 \meas(\eps) \\
+ O(h)^{\min((2\eta_j)_{j=1\ldots7},\eta_1+\eta_6)} + O(\dt_2)^{\min((2\gamma_j)_{j=1\ldots7},\gamma_1+\gamma_6)}
\end{multline*}
where $C_1'$ and $C_2'$ are zero for the scheme~\eqref{def_erreur_transmission}
In other words, we have:
\begin{multline*}
\demi \sumi \|p^{i,N_i}\|_{L^2(\O_i)}^2 + \min \left( \demi - \frac{C_3}{4}, 1 - \frac{C_1}{2}- \frac{C_1'}{2}- \frac{C_2'}{2} - \frac{C_4'}{2} - \frac{C_6}{2} \right) \sumi |p^i|_{1,\T_i,\dt_i}^2 \\\leqslant \demi \sumi \|p^{i,0}\|_{L^2(\O_i)}^2 
+ \frac{C_2+C_5}{2} \sumi \|p^i\|_{L^2(0,T;L^2(\O_i))}^2 + \frac{C_4+C_7}{2} \sumi \|p^i\|_{L^2(0,T;L^2(\G))}^2 \\
+ O(h)^{\min((2\eta_j)_{j=1\ldots7},\eta_1+\eta_6)} + O(\dt_2)^{\min((2\gamma_j)_{j=1\ldots7},\gamma_1+\gamma_6)}
\end{multline*}
We notice that $\min(2\eta_1,2\eta_6,\eta_1+\eta_6) = 2\min(\eta_1,\eta_6)$. Moreover, using a result in \cite{EGH00} (discrete Poincar\'e inequality, lemma~3.1), we have:
\begin{equation} \label{poincare}
\|p^{i}\|_{L^2(0,T;L^2(\O_i))} \leqslant \diam(\O_i) |p^{i}|_{1,\T_i,\dt_i} \quad \forall i \in \{1,2\} 
\end{equation}
We also use the discrete trace estimate proved in \cite{GHV00} :
\begin{equation} \label{trace}
 \|p^{i}\|_{L^2(0,T;L^2(\partial\O_i))}^2 \leqslant C(\O_i) \left( \|p^{i}\|_{L^2(0,T;L^2(\O_i))}^2 + |p^{i}|_{1,\T_i,\dt_i}^2 \right) \quad \forall i \in \{1,2\}
\end{equation}
Taking  small enough constants (independently of $h$) suffices to end the proof of  
Theorem~\ref{theorem}.
\end{proof}

\section{Well posedness}

\begin{theorem} \label{theorem_bien_pose}
We assume that assumption~\ref{hyp_maillage} holds. \\
Then, the problem defined by \eqref{eq_schema}-\eqref{rel_p_u} and interface scheme either \eqref{cond_trans} or  \eqref{cond_trans_rec} is well-posed.

\end{theorem}

\begin{proof} In both cases, we have a square linear system. It is thus sufficient to prove that the only solution with a zero right hand side and initial condition is zero, i.e. 
\begin{equation*}
U_\eps^{i,n} =  R_K^{i,n} =  R_{KK'}^{i,n} =  R_\eps^{i,n} =  F_K^{i,n} = p^{i,0}_K = \P_\eps = \U_\eps = 0 \quad \forall i \in \{1,2\} \quad \forall n \in \{0 \ldots N_i\} \quad \forall K \in \toi\,. 
\end{equation*}
Then, by Lemma~\ref{lemme_sans_int} and Lemma~\ref{lemme} (resp. \ref{lemme_rec}) for transmission scheme  \eqref{cond_trans} (resp. \eqref{cond_trans_rec}), we have
\begin{multline*}
\demi \sumi \sumK{i} (p_K^{i,N_i})^2 \meas(K) + \demi \sumi \sumn{i} \dt_i \sumK{i}\ \sumKp{i} \frac{(p_{K'}^{i,n+1}-p_K^{i,n+1})^2}{d(x_K,x_{K'})} \meas([KK']) \\
+ \sumi \sumn{i} \dt_i \sumK{i}\ \sumEDE{i} \frac{(p_\eps^{i,n+1}-p_K^{i,n+1})^2}{\dxy} \meas(\eps) \leqslant 0
\end{multline*}
that is for all $i=1,2$, $1\le n \le N_i$, $p^{i,n}$ is a constant. By equation~\eqref{eq_schema} in every subdomain, we then have that the value of the constant is independent of $n$. Since the initial condition is zero, the constant is actually zero. 
\end{proof}

\section{Error estimate}

Let $p^1,\,p^2$ be the solution to the continuous problem \eqref{eq_cont}-\eqref{trans_cont}. We define the interpolation on the mesh $\T_1 \cup \T_2$ at time $t_n$ by:
\begin{equation*}
\left.
\begin{aligned}
\tp_K^{i,n} =\;& p^i(x_K,t_n) \quad \forall K \in \T_i \\
\tp_\eps^{i,n} =\;& \frac{1}{\dt_i} \int_{t_{n-1/2}^i}^{t_{n+1/2}^i} \frac{1}{\meas(\eps)} \int_\eps p^i \quad \forall \eps \in \E_i \\
\tp_\eps^{i,n} =\;& 0 \quad \forall \eps \in \E_{iD} \\
\tu_\eps^{i,n} =\;& \frac{1}{\dt_i} \int_{t_{n-1/2}^i}^{t_{n+1/2}^i} \frac{1}{\meas(\eps)} \int_\eps \frac{\partial p^i}{\partial n_i} \quad \forall \eps \in \E_i\cup\E_{iD}
\end{aligned}
\quad \right\} \quad \forall n \in \{0,\ldots,N_i\} \quad \forall i \in \{1,2\}
\end{equation*}
We have to estimate the error terms $e^i_K = p^i_K-\tp_K^i$, $e^i_\eps = p^i_\eps-\tp_\eps^i$ and $q^i_\eps = u^i_\eps-\tu^i_\eps$.

\begin{theorem}
We suppose that the solution has the following regularity:
\begin{gather}
p \in C^1(0,T;C^2(\bar{\O})) \label{reg_p}
\end{gather} 
and that the numerical right hand side is such that:
\begin{gather}
\frac{1}{\dt_i} \int_{t_{n-1/2}^i}^{t_{n+1/2}^i} \frac{1}{\meas(K)} \int_K \left( f_K^{i,n} - f \right) = O(\diam(K) + \dt_i) \quad \forall n \in \{0,\ldots,N_i\} \quad \forall i \in \{1,2\} \label{hyp_F}
\end{gather} 
We assume that Assumptions~\ref{hyp_maillage}, \ref{hyp_yepsilon} hold and that transmission scheme~\eqref{cond_trans}  is used.\\
Then, we have the following estimate:
\begin{equation*}
\sumi |e^i|_{1,\T_i,\dt_i} + \sumi \|e^{i,Ni}\|_{L^2(\O_i)} = O(h+\dt)
\end{equation*}
where $\dt=\max(\dt_1,\dt_2)$. \\
The same estimate holds for transmission scheme~\eqref{cond_trans_rec} if in addition Assumption~\ref{hyp_alpha}  holds. 
\end{theorem}
\begin{proof}
It is easy to check that the errors $e_K$, $e_\eps$ et $q_\eps$ satisfy \eqref{eq_schema_gen}-\eqref{def_erreur_consistance} with error terms defined by 
\begin{itemize}
\item $\ds F_K^{i,n} = \frac{1}{\dt_i} \int_{t_{n-1/2}^i}^{t_{n+1/2}^i} \frac{1}{\meas(K)} \int_K \left( f_K^{i,n} - f \right) $
\item $\ds R_K^{i,n} = \frac{1}{\dt_i} \int_{t^i_{n-1/2}}^{t^i_{n+1/2}} \frac{1}{\meas(K)} \int_K  \frac{\partial p^i}{\partial t} - \frac{\tp_K^{i,n}-\tp_K^{i,n-1}}{\dt} $
\item $\ds R_{KK'}^{i,n} = \frac{1}{\dt_i} \int_{t_{n-1/2}^i}^{t_{n+1/2}^i} \frac{1}{\meas([KK'])} \int_{[KK']} \dpdn{i} - \frac{p^i(x_{K'},t_n)-p^i(x_K,t_n)}{d(x_K,x_{K'})} $
\item $\ds R^{i,n}_\eps  = \frac{1}{\dt_i} \int_{t_{n-1/2}^i}^{t_{n+1/2}^i} \frac{1}{\meas(\eps)} \int_\eps \dpdn{i} -\tu_\eps^{i,n}$
\item $\ds U^{i,n}_\eps = -(\tu_\eps^{i,n}-\frac{\tp_\eps^{i,n} - \tp_{K_i(\eps)}}{d(x_{K_i(\eps)},y_\eps)})$
\item $\ds \U_{\eps,T} = \Q_m(-\tu_{\eps,T}^e) - \tu_{\eps,T}^m$
\item For the scheme \eqref{cond_trans}: $\ds \P_{\eps,T} = \Q_e(\tp_{\eps,T}^m) - \tp_{\eps,T}^e$
\item  For the scheme \eqref{cond_trans_rec}: 
$\ds \P_{\eps,T} = \Q_e\left(\tp_{K_m(\eps)}^m\right) - \tp_{K_e(\eps)}^e - (d_\eps^m+d_\eps^e) \tu_{\eps,T}^e $
\end{itemize}
The derivation of the formula for $\P_{\eps,T}$ and $\U_{\eps,T}$ are made explicit in the sequel when these terms are estimated. Let us remark that we have by construction $e_K^{i,0} = 0$ for all $K\in {\cal T}_i$, $i=1,2$.

By assumption and by using Taylor expansion, it is classical to check that the error terms $R_K$, $R_{KK'}$ and $F_K$ satisfy Assumption~\ref{hyp_R} with $\eta_i=\gamma_i=1$ for $i=2,3,5$. As regards the term $U_\eps$, we proceed as in \cite{Achdou:2002:ANC}. From Assumption~\eqref{reg_p} on the regularity of the solution, we have:
\begin{gather*}
\begin{aligned}
U_\eps^{i,n} =& \frac{\frac{1}{\dt_i} \int_{t_{n-1/2}^i}^{t_{n+1/2}^i} \frac{1}{\meas(\eps)} \int_\eps p^i - p^i(x_{K_i(\eps)},t_n)}{d(x_{K_i(\eps)},y_\eps)} - \frac{1}{\dt_i} \int_{t_{n-1/2}^i}^{t_{n+1/2}^i} \frac{1}{\meas(\eps)} \int_\eps \dpdn{i} \\
=& \frac{\frac{1}{\dt_i} \int_{t_{n-1/2}^i}^{t_{n+1/2}^i} \frac{1}{\meas(\eps)} \int_\eps p^i - p^i(y_\eps,t_n)}{d(x_{K_i(\eps)},y_\eps)} + \left( \frac{p^i(y_\eps,t_n) - p^i(x_{K_i(\eps)},t_n)}{d(x_{K_i(\eps)},y_\eps)} - \dpdn{i}(y_\eps,t_n) \right) \\
&+ \left( \dpdn{i}(y_\eps,t_n) - \frac{1}{\dt_i} \int_{t_{n-1/2}^i}^{t_{n+1/2}^i} \frac{1}{\meas(\eps)} \int_\eps \dpdn{i} \right) \\
=& \frac{O(\diam(\eps))^2}{d(x_{K_i(\eps)},y_\eps)} + O(d(x_{K_i(\eps)},y_\eps)) + O(\diam(\eps)) = O(h) \text{ by assumption \ref{hyp_yepsilon}}
\end{aligned}
\end{gather*}
thus, $\eta_1=1$ and $\gamma_1\ge 1$. 
Since,
\[
R_\eps^{i,n} = \frac{1}{\dt_i} \int_{t_{n-1/2}^i}^{t_{n+1/2}^i} \frac{1}{\meas(\eps)} \int_\eps \left( \dpdn{i} - \frac{1}{\dt_i} \int_{t_{n-1/2}^i}^{t_{n+1/2}^i} \frac{1}{\meas(\eps)} \int_\eps \dpdn{i} \right) = 0\,,
\]
we have $\eta_4,\gamma_4\ge 1$. \\
We now consider the non classical consistency error terms $\P_{\eps}$ et $\U_{\eps}$. For the transmission condition \eqref{cond_trans}  we have:
\begin{equation*}
\begin{aligned}
q_{\eps}^m =& u_{\eps}^m - \tu_{\eps}^m  = \Q_m(-u_{\eps}^e) - \tu_{\eps}^m \\
=& \Q_m(-q_{\eps}^e-\tu_{\eps}^e) - \tu_{\eps}^m = \Q_m(-q_{\eps}^e) + \underbrace{\Q_m(-\tu_{\eps}^e) - \tu_{\eps}^m}_{\U_{\eps}} 
\end{aligned}
\end{equation*}
The error on the transmission condition on the interface reads:
\begin{equation*}
\begin{aligned}
e_{\eps}^e =& p_{\eps}^e - \tp_{\eps}^e  = \Q_e(p_{\eps}^m) - \tp_{\eps}^e \\
=& \Q_e(e_{\eps}^m+\tp_{\eps}^m) - \tp_{\eps}^e = \Q_e(e_{\eps}^m) + \underbrace{\Q_e(\tp_{\eps}^m) - \tp_{\eps}^e}_{\P_{\eps}}
\end{aligned}
\end{equation*}

If $(m,e)=(1,2)$, we get :
\begin{gather*}
\begin{aligned}
\P_\eps^n =& \left[\Q_2(\tp_{\eps,T}^1)\right]^n - \tp_\eps^{2,n} \\
=& \frac{\dt_1}{\dt_2} \sum_{k=1}^{K} \frac{1}{\dt_1} \int_{t_{n,k-1/2}}^{t_{n,k+1/2}} \frac{1}{\meas(\eps)} \int_\eps p^1 - \frac{1}{\dt_2} \int_{t_{n-1/2}}^{t_{n+1/2}} \frac{1}{\meas(\eps)} \int_\eps p^2 = 0
\end{aligned} \\
\begin{aligned}
\U_\eps^{n,k} =& \left[ \Q_1(-\tu_\eps^2)^{n,k} - \tu_\eps^1 \right]^{n,k} \\
=& -\tu_\eps^{2,n} - \tu_\eps^{1,n,k} \\
=& -\frac{1}{\dt_2} \int_{t_{n-1/2}}^{t_{n+1/2}} \frac{1}{\meas(\eps)} \int_\eps \dpdn{2} - \frac{1}{\dt_1} \int_{t_n^{k-1/2}}^{t_n^{k+1/2}} \frac{1}{\meas(\eps)} \int_\eps \dpdn{1} = O(\dt_2)
\end{aligned}
\end{gather*}
If $(m,e)=(2,1)$, we have 
\begin{gather*}
\begin{aligned}
\P_\eps^{n,k} =& \left[\Q_1(\tp_{\eps,T}^2)\right]^{n,k} - \tp_\eps^{1,n,k} \\ 
=& \tp_\eps^{2,n} - \tp_\eps^{1,n,k} \\
=& \frac{1}{\dt_2} \int_{t_{n-1/2}}^{t_{n+1/2}} \frac{1}{\meas(\eps)} \int_\eps p^2 - \frac{1}{\dt_1} \int_{t_n^{k-1/2}}^{t_n^{k+1/2}} \frac{1}{\meas(\eps)} \int_\eps p^1 = O(\dt_2)
\end{aligned} \\
\begin{aligned}
\U_\eps^n =& \left[ \Q_2(-\tu_\eps^1) - \tu_\eps^2 \right]^n \\
=& -\frac{\dt_1}{\dt_2} \sum_{k=1}^\K \tu_\eps^{1,n,k} - \tu_\eps^{2,n} \\
=& \frac{\dt_1}{\dt_2} \sum_{k=1}^\K \frac{1}{\dt_1} \int_{t_n^{k-1/2}}^{t_n^{k+1/2}} \frac{1}{\meas(\eps)} \int_\eps \dpdn{1} - \frac{1}{\dt_2} \int_{t_{n-1/2}}^{t_{n+1/2}} \frac{1}{\meas(\eps)} \int_\eps \dpdn{2} = 0
\end{aligned}
\end{gather*}

We consider now transmission scheme \eqref{cond_trans_rec}: $\U_{\eps}$ is left unchanged and $\P_{\eps}$ now reads :
\begin{equation*}
\begin{aligned}
e_\eps^e + d_\eps^m q_\eps^e =& p_\eps^e - \tp_\eps^e + d_\eps^m (u_\eps^e - \tu_\eps^e)  \\
=& \Q_e(p_\eps^m - d_\eps^m u_{\eps}^m) - \tp_\eps^e - d_\eps^m \tu_{\eps}^e \\
=& \Q_e(e_\eps^m + \tp_\eps^m - d_\eps^m q_{\eps}^m - d_\eps^m \tu_{\eps}^m) - \tp_\eps^e - d_\eps^m \tu_{\eps}^e \\
=& \Q_e(e_\eps^m - d_\eps^m q_{\eps}^m + d_\eps^m U_{\eps}^m) - d_\eps^e U_{\eps}^e + \underbrace{\Q_e\left(\tp_{K_m(\eps)}^m\right) - \tp_{K_e(\eps)}^e - (d_\eps^m+d_\eps^e) \tu_{\eps}^e}_{\P_{\eps}} \\
\end{aligned}
\end{equation*}

If $(m,e)=(1,2)$, we have
\begin{equation*}
\begin{aligned}
\P_\eps^n =& \left[\Q_2\left(\tp_{K_1(\eps)}^1\right)\right]^n - \tp_{K_2(\eps)}^{2,n} - (d_\eps^1+d_\eps^2) \tu_\eps^{2,n} \\
=& \frac{\dt_1}{\dt_2} \sumk p^1(x_{K_1(\eps)},t_n^k) - p^2(x_{K_2(\eps)},t_n) - (d_\eps^1+d_\eps^2) \frac{1}{\dt_2} \int_{t_{n-1/2}^i}^{t_{n+1/2}^i} \frac{1}{\meas(\eps)} \int_\eps \dpdn{2} \\
=& \frac{1}{\dt_2} \sumk \int_{t_n^{k-1/2}}^{t_n^{k+1/2}} \frac{1}{\meas(\eps)} \int_\eps \left( p^1(x_{K_1(\eps)},t_n^k) - p^2(x_{K_2(\eps)},t_n) - d(x_{K_1(\eps)},x_{K_2(\eps)}) \dpdn{2} \right) \\
=& O(h+\dt)
\end{aligned}
\end{equation*}

If $(m,e)=(2,1)$, $\P_{\eps}$ reads :
\begin{equation*}
\begin{aligned}
\P_\eps^{n,k} =& \left[\Q_1\left(\tp_{K_2(\eps)}^2\right)\right]^{n,k} - \tp_{K_1(\eps)}^{1,n,k} - (d_\eps^1+d_\eps^2) \tu_\eps^{1,n,k} \\
=& \frac{1}{\dt_1} \int_{t_n^{k-1/2}}^{t_n^{k+1/2}} \frac{1}{\meas(\eps)} \int_\eps \left( p^2(x_{K_2(\eps)},t_n) - p^1(x_{K_1(\eps)},t_n^k) - d(x_{K_1(\eps)},x_{K_2(\eps)})  \dpdn{1} \right) \\
=& O(h+\dt)
\end{aligned}
\end{equation*}
For all cases, we have thus $\eta_6, \eta_7, \gamma_6, \gamma_7 \ge 1$. Then the error estimate follows from Theorem~\ref{theorem}. 
\end{proof}


\section{Numerical results}
\label{sec:resnum}
 
In this part, we illustrate the method with a parabolic equation coming from a previous article of Ewing and Lazarov  \cite{EL94}.
We consider the (IS$_2$) interface conditions, i.e. equation (\ref{eq:master2_overlap}) in one dimension, which are more natural. 
We solve the following  model problem :
\begin{eqnarray}
\frac{\partial }{\partial t}p(x,t) - \frac{\partial^2 p}{\partial x^2}(x,t) &=& f(x,t)  \quad \forall t \in [0.,0.1] \quad \forall x \in [0.,1.] \\ 
p(x,t) &=& 0  , \forall x \in \partial \Omega\ \\
p(x,0) &=& 0
\end{eqnarray}
The following function is used as an exact solution :
\[ p(x,t) = exp(20(t-t^2) - 37 x^2 + 8x - 1) \]
\begin{figure}[h]
  \centering
  \includegraphics[width=0.45\linewidth]{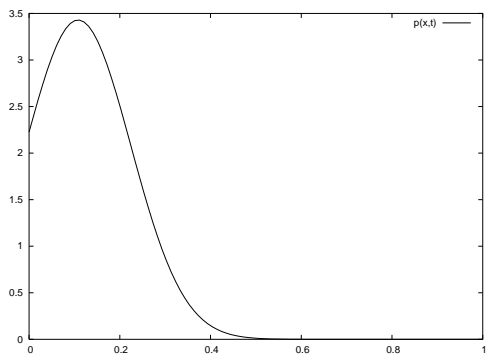}
  \caption{Exact solution for t=0.1}
\end{figure}
This function represents a bump with a maximum value near  the position $x=0.15$. In the interval $[0.5, 1.]$, the function is  close to 0. In this interval, the function changes negligibly in time. In contrast, the function changes rapidly in time in the interval $[0.,0.5]$ and simulates a local behavior.

We use two different time step sizes : 
\begin{itemize}
\item  a fine time step  $\delta t_1 = 0.002$  in the subdomain $\Omega_1 = [0.,0.25]$, discretized with a fine grid $\delta x_1 = 0.01$
\item a coarse time step $\delta t_2 = 0.02$   in the subdomain $\Omega_2 =[0.25,1.]$, discretized with a coarse grid $\delta x_2 = 0.05$
\end{itemize}

The interface is placed at $x=0.25$ . This is a worst case since  the domain with  local refinement only partially covers the interval $[0.,0.25]$  where the solution changes quickly.

We consider two cases.  The first one (coarse master)  is when the coarse domain enforces the  Dirichlet condition, see equation (\ref{eq:master2_overlap}).
 The second one (fine master) is when the refined  domain enforces  the Dirichlet condition, see equation (\ref{eq:master1_overlap}). 
We make a comparison with the  algorithm given by Mlacnik and Heinemann \cite{M02,MH01}. 
In the following pictures, we plot  the evolution of the errors in space and the time evolution of the $L^2$ norm of the error
between the exact solution, the two local time step methods and the solution with the fine or coarse time step on the whole domain.

At each coarse time step, we solve the set of discretized equations using the iterative algorithm explained in section \ref{section:SolutionMethod} with the stopping criterium  $\varepsilon = 10^{-5}$. In the following figures, we plot the error between these two solutions and the exact solution. For completeness, we also plot the error
for a computation with either the coarse or the fine time step on the whole domain.  The number of iterations needed to reach the convergence is quite small; it is about 6 for the fine master method and about 8 for the coarse master method.

We notice that for both cases, the error is significantly smaller than the one of the coarse time step. Morever, for the fine master method, see equation(\ref{eq:master1_overlap}), the error is close to the fine time step error in the refined zone, see figure \ref{err::it}. 

As explained in section \ref{section:SolutionMethod}, it is not necessary to iterate until convergence the algorithm to obtain a conservative method. In figure \ref{err::one}, we plot the error after only one iteration of the corrector stage. As expected,the errors are larger than with the converged solutions.
Let us recall, that the method proposed by Mlacnik  corresponds to the fine master curve in figure \ref{err::one}.

\begin{figure}[htp]
\centering
\includegraphics[angle=-90,width=.45\linewidth]{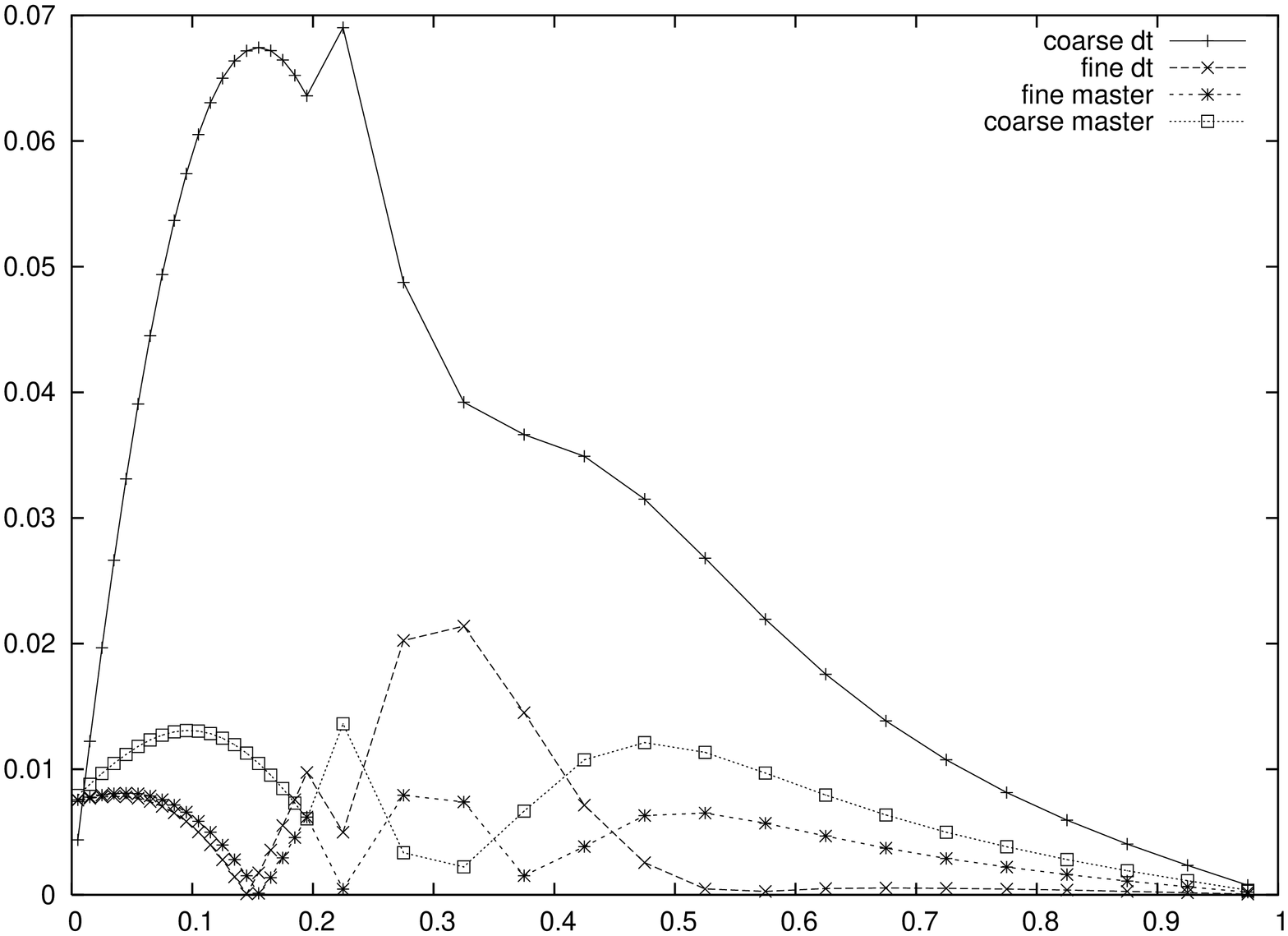}
\includegraphics[angle=-90,width=.45\linewidth]{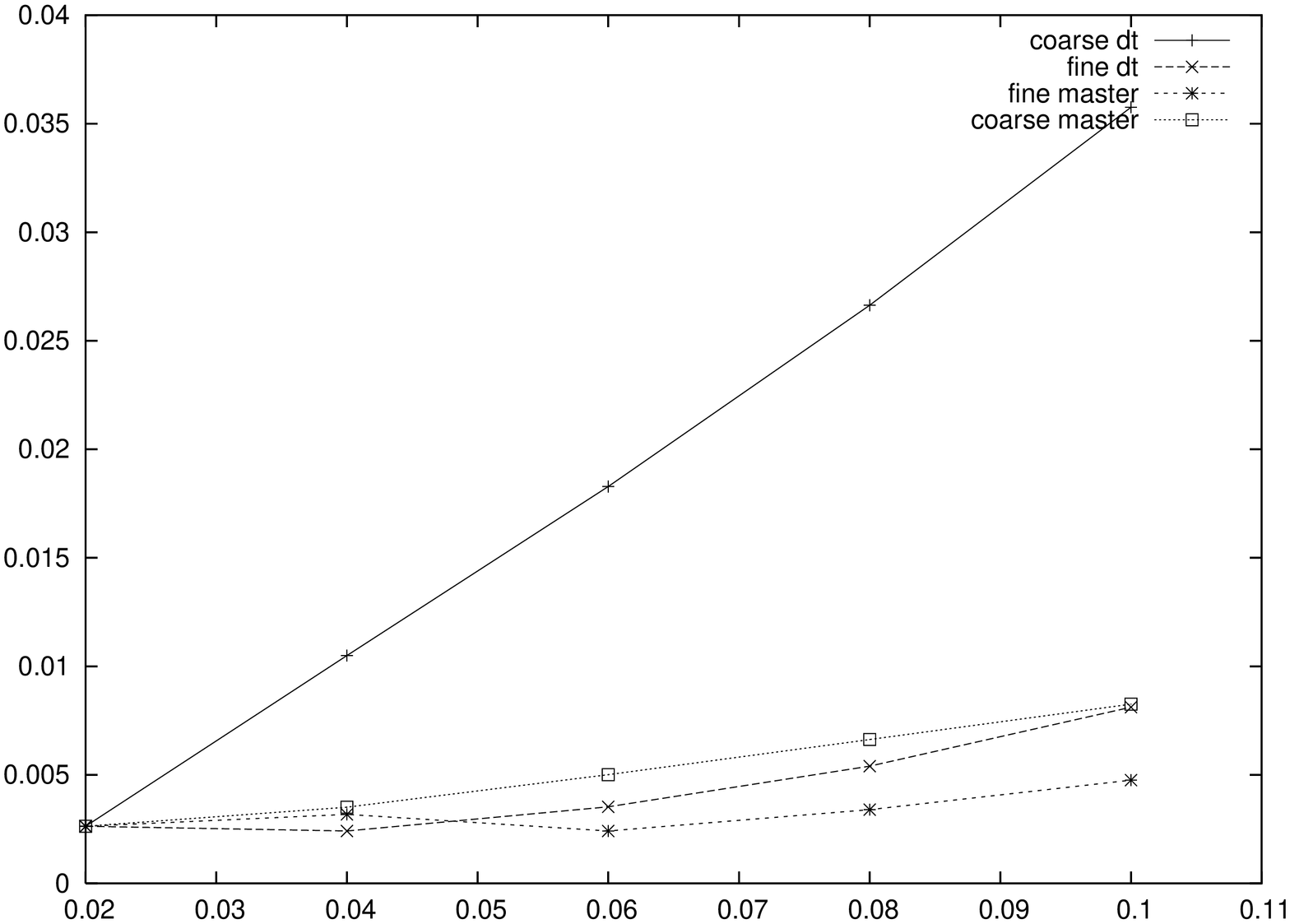}
\caption{After convergence, on left : error in space with the exact solution at time  $t=0.1$, on right : time evolution of the $L^2$ norm of the error.}\label{err::it}
\includegraphics[angle=-90,width=.45\linewidth]{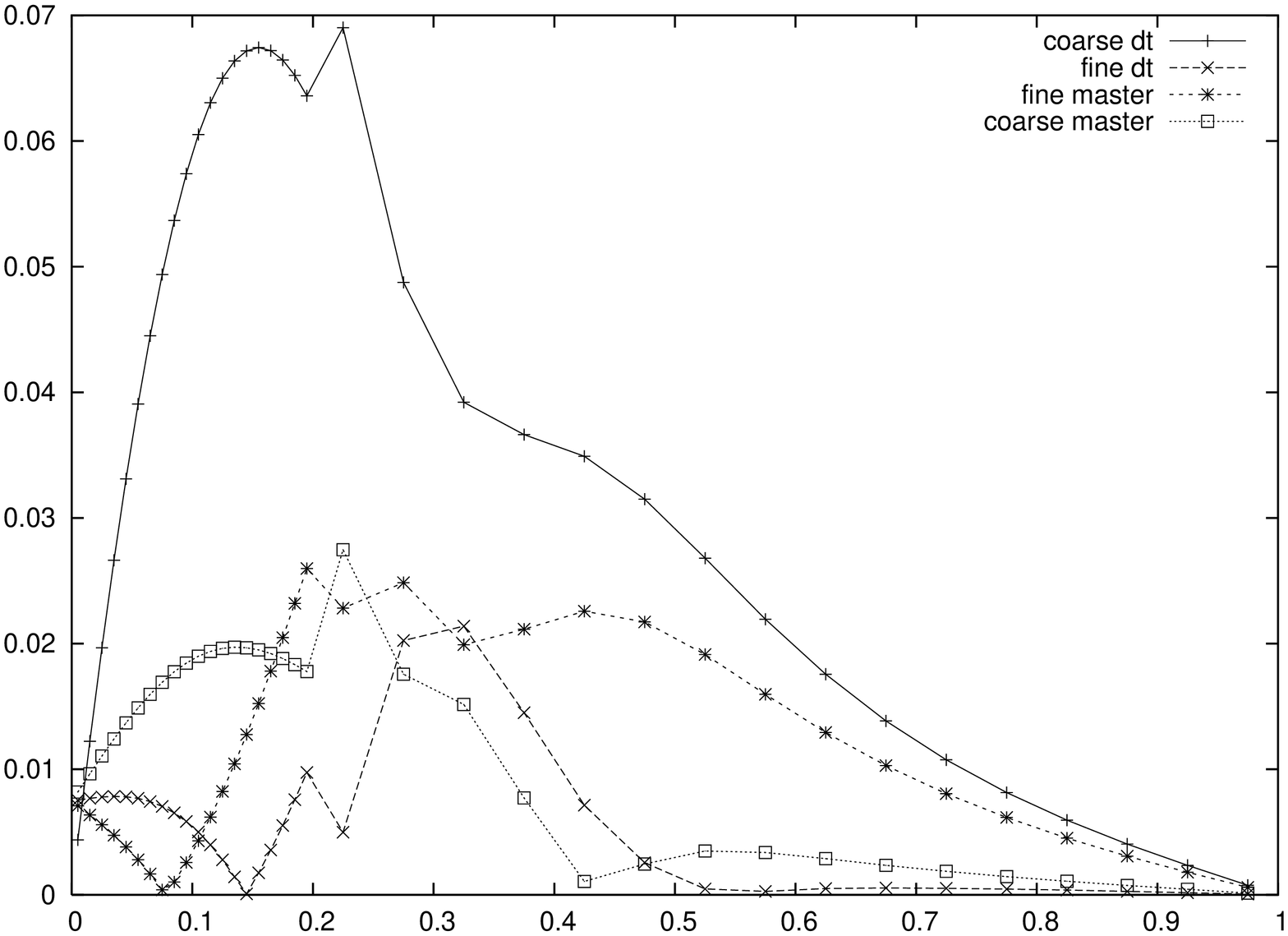}
\includegraphics[angle=-90,width=.45\linewidth]{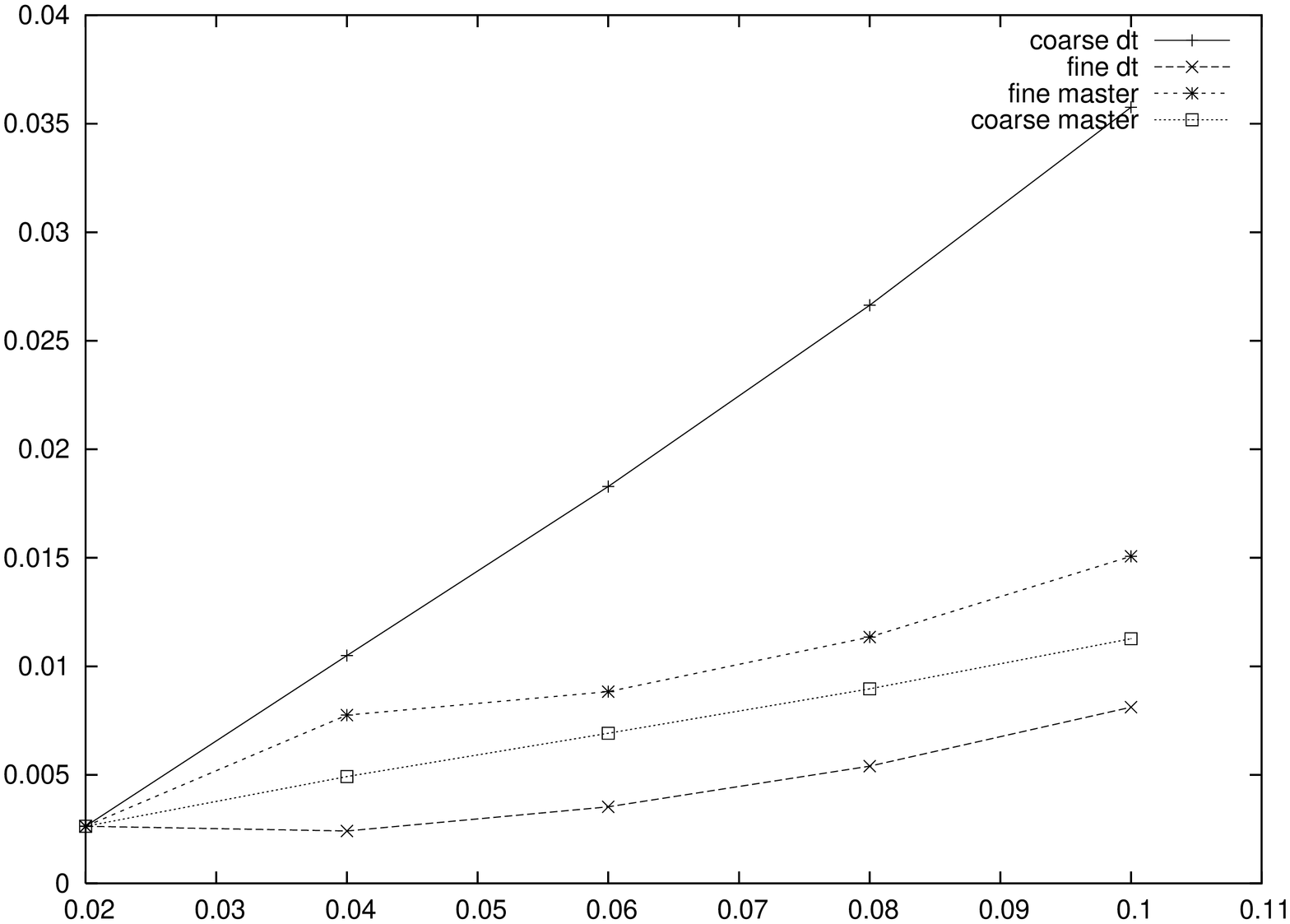}
\caption{After one iteration, on left : error in space with the exact solution at time  $t=0.1$, on right :  time evolution of the $L^2$ norm of the error.}\label{err::one}

\end{figure}


\section{Conclusion}
\label{sec:conclusion}

We have proposed a local time step strategy for solving problems on grids with different time steps in different regions. We have analyzed two schemes: \eqref{cond_trans} and \eqref{cond_trans_rec}. In \eqref{cond_trans}, the coupling involves additional interface unknowns. Scheme~\eqref{cond_trans_rec} is written only in terms of ``classical'' finite volume unknowns. Both schemes are conservative, of order one in space and time. The assumptions are more restrictive for \eqref{cond_trans_rec} than for \eqref{cond_trans}. We have presented an iterative solution method for solving the composite grid system. Its main feature is that at every stage, conservativity is ensured. Numerical tests on a toy problem confirm the capabilities of the method. The scheme is being implemented in a multiphase three dimensional simulation code.


:
\bibliographystyle{plain}
\bibliography{localtimestep}

\end{document}